\pdfoutput=1
\documentclass[11pt]{article}

\usepackage[letterpaper,margin=1in]{geometry}
\usepackage{amsthm,amsmath,amsfonts,amssymb}
\usepackage[numbers,sort&compress]{natbib}
\usepackage[colorlinks=true,linkcolor=blue,citecolor=blue,urlcolor=blue]{hyperref}

\hypersetup{
  pdftitle={Quadratic mixing for the Lee--Sidford Dikin walk},
  pdfauthor={Yiping Lu},
  pdfkeywords={Dikin walk, Lee--Sidford metric, Lewis weights,
    average self-concordance, Wiener chaos}
}
\numberwithin{equation}{section}
\theoremstyle{plain}
\newtheorem{theorem}{Theorem}[section]
\newtheorem{proposition}[theorem]{Proposition}
\newtheorem{lemma}[theorem]{Lemma}
\newtheorem{corollary}[theorem]{Corollary}
\theoremstyle{definition}
\newtheorem{definition}[theorem]{Definition}
\theoremstyle{definition}
\newtheorem{remark}[theorem]{Remark}

\newcommand{\R}{\mathbb R}
\newcommand{\C}{\mathbb C}
\newcommand{\E}{\mathbb E}
\newcommand{\Pp}{\mathbb P}
\newcommand{\Id}{I}
\newcommand{\Diag}{\operatorname{Diag}}
\newcommand{\diag}{\operatorname{diag}}
\newcommand{\tr}{\operatorname{tr}}
\newcommand{\op}{\mathrm{op}}
\newcommand{\HS}{\mathrm{HS}}
\newcommand{\polylog}{\operatorname{polylog}}
\newcommand{\cE}{\mathcal E}
\newcommand{\cF}{\mathcal F}
\newcommand{\cP}{\mathcal P}
\newcommand{\cR}{\mathcal R}
\newcommand{\cT}{\mathcal T}
\newcommand{\cX}{\mathcal X}
\newcommand{\cY}{\mathcal Y}
\newcommand{\ip}[2]{\langle #1,#2\rangle}
\newcommand{\norm}[1]{\left\lVert #1\right\rVert}
\newcommand{\abs}[1]{\left\lvert #1\right\rvert}
\newcommand{\dd}{\,\mathrm d}
\newcommand{\had}{\mathbin{\circ}}
\newcommand{\eps}{\varepsilon}
\newcommand{\transpose}{\mathsf T}
\DeclareMathOperator{\Tr}{Tr}

\begin{document}

\title{Quadratic mixing for the Lee--Sidford Dikin walk}
\author{Yiping Lu\\
\small Department of Industrial Engineering and Management Sciences\\
\small Northwestern University\\
\small \texttt{yiping.lu@northwestern.edu}}
\date{}
\maketitle

\begin{abstract}
Let $K=\{x\in\R^d:Ax\ge b\}$ be a bounded full-dimensional
polytope with $m$ inequalities. For targets on $K$ with density
proportional to $\exp(-\vartheta^\transpose x)$, we prove a
$\widetilde O(d^2)$ bound, measured in Markov chain iterations, for
the chi-square mixing of the $1/2$-lazy Metropolis Dikin walk based
on a polylogarithmically rescaled Lee--Sidford metric. The estimate
holds for every initial law with finite chi-square divergence. The
proof controls the reverse Gaussian proposal through a quantitative
average self-concordance estimate at the required $O(1/d)$ scale.
We establish this estimate from uniform analytic derivative bounds
for the Lewis-weight subspace and an all-order projection-weighted
contraction inequality for Gaussian chaoses. For order $j$, the
latter bounds each base-point polynomial at scale $d^{j/2}$, up to
$j$-dependent and logarithmic factors. Truncating at logarithmic
order then absorbs the remaining pathwise loss. \textbf{This resolves, in the affirmative, the open problem of whether the Gaussian
Lee--Sidford Dikin walk mixes in $\widetilde O(d^2)$ iterations: the
previously stated quadratic bounds \cite[Theorem~1.4]{LLV20},
\cite[Theorem~1.2]{GKM24} rest on reverse-proposal steps identified as
unresolved in \cite[Section~1]{Kook26}, whose $\widetilde O(d^{9/4})$ bound
was the best rigorously established rate prior to this work. The proof is found by ChatGPT 5.6 Pro. The author verified the proof and promised to fine-tune the draft later.}
\end{abstract}



\section{Introduction}

Interior-point geometry supplies affine-invariant local coordinates
for both optimization and sampling. At its current point, a Dikin
walk proposes a Gaussian increment with covariance proportional to
the inverse of a barrier metric and then applies a Metropolis
correction. The metric adapts to the boundary, but the same state
dependence creates the principal probabilistic difficulty: the
reverse Gaussian proposal is evaluated in a random metric at the
proposed point.

The logarithmic-barrier Dikin walk has an iteration bound depending
polynomially on the number of inequalities \cite{KN12,Nar16,SV16}.
Lewis weights originate in functional analysis \cite{Lewis78}, became
algorithmic row-sampling and barrier tools \cite{CP15,LS14,LS19}, and
subsequently entered local sampling geometries \cite{CDWY18,LLV20}.
We ask whether the resulting Lee--Sidford local-metric chain has a
mixing bound quadratic in the ambient dimension, up to logarithmic
factors.

\subsection{Problem setting and main results}

Let $K=\{x\in\R^d:Ax\ge b\}$ be bounded and full-dimensional, where
$A\in\R^{m\times d}$ has full column rank and no zero row. The target
distributions are
\begin{equation}
 \pi_\vartheta(\dd x)
 =Z_\vartheta^{-1}e^{-\vartheta^\transpose x}\mathbf 1_K(x)\dd x,
 \qquad \vartheta\in\R^d,
 \label{eq:target}
\end{equation}
and include the uniform distribution when $\vartheta=0$.
For probability measures $\mu\ll\pi$, write
\begin{equation}
 \chi^2(\mu\Vert\pi)
 :=\int\left(\frac{\dd\mu}{\dd\pi}-1\right)^2\dd\pi.
 \label{eq:chisquare}
\end{equation}
Thus finite chi-square divergence is exactly the condition
$\dd\mu/\dd\pi-1\in L^2(\pi)$.

Given a local metric $g$ and a proposal radius $\varrho>0$, write
\[
 Q_g^{\varrho}(x,\cdot)=N\!\left(x,\frac{\varrho^2}{d}g(x)^{-1}\right).
\]
The Metropolis Dikin kernel, with the usual Metropolis--Hastings
correction \cite{Hastings70}, proposes
$y\sim Q_g^{\varrho}(x,\cdot)$, rejects unless
$y\in\operatorname{int}K$, and otherwise accepts with probability
\begin{equation}
 1\wedge\frac{\pi_\vartheta(y)q_y(x)}
 {\pi_\vartheta(x)q_x(y)},
 \qquad
 q_x(y)=\frac{\det(g(x))^{1/2}}{(2\pi\varrho^2/d)^{d/2}}
 \exp\!\left(-\frac{d}{2\varrho^2}\norm{y-x}_{g(x)}^2\right).
 \label{eq:DikinKernel}
\end{equation}
The $1/2$-lazy version stays at its current state with probability
$1/2$ and otherwise applies this kernel.

The precise question is whether a polylogarithmic rescaling of the
Lee--Sidford metric, together with a dimension-independent proposal
radius, makes this lazy kernel mix in $\widetilde O(d^2)$ iterations,
uniformly over $\vartheta$, from every initial law whose density
belongs to $L^2(\pi_\vartheta)$. The obstruction is not the forward
Gaussian proposal itself but the random quadratic form in its reverse
density. We must control that form at scale $O(\varrho^2/d)$ with high
probability and then obtain overlap of the accepted densities, without
assuming a function whose Hessian is comparable to $g$. The following
quantitative average self-concordance condition isolates the required
probabilistic input.

\begin{definition}[Quantitative average self-concordance]
\label{def:qasc}
For a metric $g$ and $\eps\in(0,1)$, its quantitative average
self-concordance (ASC) radius $R_\eps(g)$ is the
supremum of radii $r\ge0$ such that, for every
$x\in\operatorname{int}K$, every $0<s\le r$, and
$z\sim N(x,s^2g(x)^{-1}/d)$,
\begin{equation}
 \Pp\!\left(
 \abs{\norm{z-x}_{g(z)}^2-\norm{z-x}_{g(x)}^2}
 \le\frac{2\eps s^2}{d},\ z\in\operatorname{int}K
 \right)\ge1-\eps.
 \label{eq:quantASCdef}
\end{equation}
\end{definition}

The ASC condition in \cite{KV24} is one-sided and does not place an
absolute value around the metric discrepancy. Definition~2.1 of
\cite{Kook26} makes the two-sided event and the interior convention
explicit. We adopt those two conventions and package the dependence
on $\eps$ by introducing the maximal admissible radius $R_\eps(g)$.
Only the one-sided implication is needed in the mixing reduction; the
argument below proves the two-sided form.

\begin{theorem}[Quantitative average self-concordance]
\label{thm:asc}
There are universal constants $c_{\rm asc},C_{\rm asc}>0$ such that
for every $p\ge4$ and every $\eps\in(0,1/4)$, the baseline
$\ell_p$ Lee--Sidford metric $g_0$ defined in \eqref{eq:g0} satisfies
\begin{equation}
 R_\eps(g_0)\ge
 \underline r_\eps
 :=c_{\rm asc}\eps^3p^{-C_{\rm asc}}
 \left(1+\log\frac{2(m+d+1)}{\eps}\right)^{-C_{\rm asc}}.
 \label{eq:reps}
\end{equation}
The same constants work uniformly for all $m,d,p,\eps,A$, and $b$.
\end{theorem}

Set
\begin{equation}
 p_m:=4+2\log(2m),\qquad
 \lambda_m:=(1+p_m^2)(1+p_m),\qquad
 \eps_0:=\frac1{20},
 \label{eq:pchoice}
\end{equation}
and define
\begin{equation}
 \widehat r_0
 :=\min\!\left\{1,\,
 c_{\rm asc}\eps_0^3p_m^{-C_{\rm asc}}
 \left(1+\log\frac{2(m+d+1)}{\eps_0}\right)^{-C_{\rm asc}}
 \right\},
 \quad
 L_{m,d}:=\lambda_m\widehat r_0^{-2}.
 \label{eq:Lchoice}
\end{equation}

\begin{theorem}[Quadratic mixing in Markov chain iterations]
\label{thm:main}
Let $g_0$ be the baseline metric defined in \eqref{eq:g0}, with
$p=p_m$,
and set $g_\star=L_{m,d}g_0$. There is a proposal radius
$\varrho_\star>0$, depending only on universal constants, such that
for every target \eqref{eq:target}, every
$\mu_0\ll\pi_\vartheta$ with
$\chi^2(\mu_0\Vert\pi_\vartheta)<\infty$, and every
$\delta\in(0,1)$, the law $\mu_T$ of the $1/2$-lazy walk after any
integer $T\ge0$ satisfies
$\chi^2(\mu_T\Vert\pi_\vartheta)\le\delta$ whenever
\begin{equation}
 T\ge C d^2L_{m,d}
 \log\!\left(
 \frac{1+\chi^2(\mu_0\Vert\pi_\vartheta)}{\delta}
 \right).
 \label{eq:mainmix}
\end{equation}
Moreover, $L_{m,d}=\polylog(m,d)$.
\end{theorem}

Theorem~\ref{thm:main} counts Markov chain iterations only. It excludes
the cost of computing exact or approximate Lewis weights
\cite{CP15,FLPS22} and makes no membership- or evaluation-oracle
complexity claim. No affine rounding of $K$ is required. The sole
initialization condition is finite chi-square divergence. In
particular, if
$\dd\mu_0/\dd\pi_\vartheta\le M$, then
$1+\chi^2(\mu_0\Vert\pi_\vartheta)\le M$, so the displayed bound is
$Cd^2L_{m,d}\log(M/\delta)$.

\subsection{Prior work and scope}

\paragraph*{Barrier metrics for local sampling.}
For a polytope with $m$ inequalities in dimension $d$, the
logarithmic-barrier Dikin walk has a
$\widetilde O(md)$ warm-start iteration bound for uniform sampling
\cite{KN12,SV16}; Narayanan extended the barrier-based viewpoint to
more general convex sets and randomized interior-point methods
\cite{Nar16}. Lewis weights originate in Banach-space geometry
\cite{Lewis78}. Their algorithmic computation was developed in
\cite{CP15,FLPS22}, while approximate John and Lewis-weight barrier
constructions for linear programming were developed in
\cite{LS14,LS19}. These weighted geometries were subsequently used
to construct local-metric sampling algorithms \cite{CDWY18,LLV20}.

Chen, Dwivedi, Wainwright and Yu proved that the Vaidya walk mixes
from an $M$-warm start in
$\widetilde O(m^{1/2}d^{3/2})$ iterations
\cite[Theorem~1]{CDWY18}. Under the additional condition
$m\le\exp(\sqrt d)$, their Gaussian walk based on approximate John
weights has mixing time $\widetilde O(d^{5/2})$
\cite[Theorem~2]{CDWY18}. They conjectured a
$\widetilde O(d^2)$ bound for a differently scaled improved John
walk \cite[Conjecture~5]{CDWY18}. These results concern the uniform
law, total-variation error and an $M$-warm initialization. Our kernel
instead uses the $\ell_p$ Lee--Sidford metric, treats every
exponential tilt, and proves chi-square contraction from every
finite-chi-square initial law; it therefore does not resolve the
improved-John conjecture. The exact John-ellipsoid walk of Gustafson
and Narayanan is a distinct algorithm: it samples uniformly from the
maximum-volume ellipsoid of the symmetrized body and has a
$\widetilde O(d^7)$ warm-start bound \cite{GN23}.

\paragraph*{Quadratic Lee--Sidford claims and the reverse proposal.}
Laddha, Lee and Vempala stated a $\widetilde O(d^2)$ bound for their
Lee--Sidford Dikin walk \cite[Theorem~1.4]{LLV20}; their proposal is
uniform on a local ellipsoid rather than Gaussian. Gu et al. stated a
robust Lee--Sidford bound
\[
 \widetilde O\!\left(
 (d^2+dL^2R^2)\log\frac{w}{\delta}
 \right)
\]
for an $L$-Lipschitz potential on a body contained in a radius-$R$
ball, from a $w$-warm start \cite[Theorem~1.2]{GKM24}. Kook
\cite[Section~1]{Kook26} identifies unresolved reverse-proposal steps
in the published proofs of these bounds: in the former argument the
current point must be controlled in the reverse Dikin ellipsoid,
while the proof sketch of \cite[Lemma~D.5]{GKM24} does not establish
the $O(\varrho^2/d)$ quadratic-form scale required by average
self-concordance. We therefore record these bounds as previously
stated results but do not use them as inputs. This issue concerns only
the reverse-density estimate; the SSC, log-determinant convexity and
symmetry lemmas from \cite{LLV20} remain independent inputs below.

For the same Gaussian Lee--Sidford Dikin chain considered here, the
recent preprint \cite[Theorem~1.1]{Kook26} proves average
self-concordance after a
$\widetilde O(d^{1/4})$ metric rescaling and obtains
$\widetilde O(d^{9/4})$ iterations for exponential targets. It also
introduces the selective higher-order expansion, moving orthonormal
frames and Wiener-chaos organization used here, and observes that
order-$k$ estimates would lead to exponent $2+1/(k+1)$
\cite[Section~1.3]{Kook26}. Our contribution is to supply all-order
analytic frame and contraction estimates, allowing the rescaling loss
to be reduced to polylogarithmic size.

\paragraph*{Other local-metric samplers.}
Jiang and Chen prove $\widetilde O(md+\kappa d)$ iterations for the
soft-threshold Dikin walk and
$\widetilde O(d^{5/2}+\kappa d)$ for their regularized Lewis-weight
Dikin walk, in the notation and under the conditions of that work,
with extensions to weakly log-concave truncations and
beyond-worst-case bounds \cite{JiangChen25}. Gatmiry, Kelner and
Vempala analyze Riemannian HMC using a hybrid of the Lewis-weight and
logarithmic barriers and obtain $\widetilde O(m^{1/3}d^{4/3})$
mixing for polytope sampling \cite{GKV24}. These are different
kernels, and their complexity expressions are not directly
comparable to the iteration-only bound proved here.

\paragraph*{Oracle samplers and global geometry.}
For a convex body specified by a well-defined membership oracle
$B_1(x_0)\subseteq K\subseteq B_D(x_0)$, the restart version of
In-and-Out, initialized from an $M$-warm law, achieves for every
$q\ge2$ a finite-order R\'enyi guarantee using
\[
 \widetilde O\!\left(
 qMd^2\Lambda\log^6\frac1\epsilon
 \right),
 \qquad
 \Lambda:=\norm{\operatorname{cov}\pi}_{\op},
\]
expected membership queries
\cite[Theorem~3 and Corollary~4]{KVZ26}. For an arbitrary log-concave
density $\pi\propto e^{-V}$ given by an evaluation oracle,
\cite[Theorem~1.1]{KVUnified26} assumes
\[
 B(0,1)\subseteq\{x:V(x)-\inf V\le10d\},
 \qquad
 M_q:=\norm{\dd\pi_0/\dd\pi}_{L^q(\pi)}\le10,
\]
and gives
$\widetilde O(qd^2\Lambda\log^3(1/\epsilon))$ expected evaluation
queries, where
\[
 q\ge\max\!\left\{2,
 \widetilde\Omega\!\left(
 \log\bigl(d^2\Lambda\log(1/\epsilon)\bigr)
 \right)\right\}.
\]
Algorithmic diffusion and its subsequent cold-start improvements are
developed in \cite{KVDiff25,KVCold25}. Lee and Vempala prove an
$O(\sqrt d)$ Poincar\'e bound for isotropic log-concave measures and
derive a $\widetilde O(d^{5/2})$ warm-start bound for the ball walk
\cite{LeeVempala24}; this is global Euclidean isoperimetric context,
distinct from the local barrier geometry used here.

These oracle results concern different algorithms and include global
normalization and covariance parameters. Our result instead counts
iterations of an affine-invariant local chain given an explicit
inequality description, requires no global rounding, and excludes the
arithmetic cost of computing the Lee--Sidford metric. It is therefore
neither an oracle-complexity comparison nor a claim to be the first
quadratic-dimension sampler for general log-concave targets.

\subsection{Proof strategy and contributions}

The proof first closes the abstract mixing bridge for the present
linear targets. Lemma~\ref{lem:linearTargetOverlap} works directly
with the common accepted density, rather than with two separate
rejection estimates. It constructs good events for both starting
points, controls the covariance and displaced-mean terms in the
comparison of $q_x$ and $q_y$, intersects the events with a common
target halfspace, and keeps the interior condition explicit. This is
the step that removes the function-counterpart hypothesis used in the
general framework of \cite{KV24}.

The ASC proof then follows the higher-order route of
\cite[Sections~1.2--1.3 and~3]{Kook26}. That work introduced the
selective bottleneck recursion, moving orthonormal frames, and the
Wiener-chaos organization, and showed how fixed-order estimates improve
the dimension exponent. Our additions are uniform analytic frame
jets, an all-order projection-weighted Wick-contraction theorem, and a
logarithmic-order summation of the selective recursion. Here $H_j$ is
the $j$th bottleneck polynomial defined in \eqref{eq:Hj}, and
$\rho_p\asymp p^{-3}$ is the analytic radius in \eqref{eq:rhop}. A
direct pathwise bound gives $|H_4(t)|\lesssim d^{5/2}$, whereas at the
base point we prove
\begin{equation}
 \norm{H_4(0)}_{L^2}\le p^{O(1)}d^2,
 \qquad
 \norm{H_j(0)}_{L^2}
 \le p^2(Cj)^{Cj}\rho_p^{-(j-1)}(d+j)^{j/2}
 \quad(j\ge2).
 \label{eq:introHj}
\end{equation}
The natural $d^{j/2}$ scale is used only over the logarithmic range of
$j$ required by the proof.

The raw-to-Wick expansion, chaos orthogonality and product formula for
multiple Wiener integrals are classical
\cite{Janson97,Nualart06,PeccatiTaqqu11}. Our new probabilistic input
is a projection-weighted contraction estimate adapted to a Parseval
frame. If $U^\transpose U=I_d$, $P=UU^\transpose$ and
$W=\Diag(\diag P)$, then
\[
 P\preceq I_m,\qquad
 W^{-1/2}P^{\had2}W^{-1/2}\preceq I_m,\qquad
 \sum_iw_i=d.
\]
The first two identities control the zero and one contractions. After
$r$ Gaussian traces, the double contraction has size
$d^{r+1}d^{(n-2r-2)/2}=d^{n/2}$, avoiding the extra factor
$d^{1/2}$. This is a standalone weighted Gaussian-polynomial
statement under projection and relative-row hypotheses, not a theorem
for arbitrary weights.

The proof follows one dependency chain. Proposition~\ref{prop:framework}
reduces Theorem~\ref{thm:main} to strong and lower-trace
self-concordance, symmetry, and a dimension-free ASC modulus for the
scaled metric. Propositions~\ref{prop:knownregularity} and
\ref{prop:scaling} verify the deterministic hypotheses, so the main
task is Theorem~\ref{thm:asc}. Lemma~\ref{lem:selective-expansion}
decomposes the metric discrepancy into base-point terms, controlled
path terms, and one terminal remainder. Theorem~\ref{thm:jets}
supplies all-order derivatives of $N_x$; together with
Theorem~\ref{thm:weightedWick}, it yields the base-point estimate in
Theorem~\ref{thm:Hbase}. Lemma~\ref{lem:pathbounds} controls the
remaining path terms, and Lemma~\ref{lem:constants} chooses a
logarithmic truncation order that absorbs the terminal $\sqrt d$ loss.
Lemma~\ref{lem:scaledASC} and Proposition~\ref{prop:framework} then
complete the proof of Theorem~\ref{thm:main}.

\paragraph*{Notation.}
Unless a subscript indicates otherwise, vector norms are Euclidean.
We use $\norm{\cdot}_{\op}$, $\norm{\cdot}_{\mathrm F}$, and
$\norm{\cdot}_{\HS}$ for operator, Frobenius, and Hilbert--Schmidt
norms, and $\norm{u}_{g(x)}^2=u^\transpose g(x)u$ for a local metric
norm. The symbols $d_{\rm TV}$ and $D_{\rm KL}$ denote total variation
distance and relative entropy. The identity on $\R^k$ is $\Id_k$;
$u^\transpose$ is transpose; $\Diag(u)$ forms a diagonal matrix and
$\diag(M)$ extracts its diagonal; and $M^{\had2}=M\had M$ denotes the
Hadamard square. For symmetric matrices, $B\preceq C$ means that
$C-B$ is positive semidefinite. The derivative $D^kF(x)$ is viewed as
a symmetric $k$-linear map. We write $X\lesssim Y$ for $X\le CY$ with
a universal $C$, and $X\asymp Y$ when both comparisons hold. Matrix
trace is $\tr$; the iterated tensor trace $\Tr$ is defined in
Section~\ref{sec:wick}. Constants $C,c>0$ are universal and may change
from line to line. Finally, $\polylog(m,d)$ denotes a quantity bounded
by $C_0(1+\log(m+d+1))^{C_1}$ for universal $C_0,C_1>0$, which may
differ between occurrences.

\section{Lee--Sidford geometry and the mixing reduction}\label{sec:prelim}

This section fixes the Lee--Sidford metric, proves the abstract
accepted-density overlap and conductance reduction, and then verifies
the deterministic Lee--Sidford hypotheses. The remaining random input,
quantitative ASC, is proved in Sections~\ref{sec:selective}--\ref{sec:path}.

\subsection{The metric and Lewis calculus}

Let
\[
 K=\{x\in\R^d:Ax\ge b\}
\]
be bounded and full-dimensional. For $x\in\operatorname{int}K$, define
\[
 S_x=\Diag(Ax-b),\qquad A_x=S_x^{-1}A.
\]
Fix $p\ge4$ and put
\[
 \beta=\frac12-\frac1p,\qquad \alpha=2\beta=1-\frac2p,
 \qquad c_p=1-\frac2p.
\]
The $\ell_p$ Lewis weights $w_x\in\R_{++}^m$ are the unique positive
vector satisfying \cite[Section~5]{LS19}
\begin{equation}
 w_x=\diag\!\left(W_x^\beta A_x
 (A_x^\transpose W_x^{2\beta}A_x)^{-1}
 A_x^\transpose W_x^\beta\right),
 \qquad W_x=\Diag(w_x).
 \label{eq:lewis}
\end{equation}
The baseline Lee--Sidford metric is
\begin{equation}
 g_0(x)=A_x^\transpose W_x^{2\beta}A_x.
 \label{eq:g0}
\end{equation}
Set
\begin{equation}
 P_x=W_x^\beta A_xg_0(x)^{-1}A_x^\transpose W_x^\beta,
 \qquad \Lambda_x=W_x-P_x^{\had2},
 \label{eq:P-Lambda}
\end{equation}
where $P^{\had2}=P\had P$, and define
\begin{equation}
 \bar\Lambda_x=W_x^{-1/2}\Lambda_xW_x^{-1/2},
 \qquad
 N_x=2\bar\Lambda_x(\Id-c_p\bar\Lambda_x)^{-1}.
 \label{eq:Ndef}
\end{equation}
The standard Lewis identities \cite[Section~2.2]{Kook26} give
\begin{equation}
 P_x^{\had2}\preceq W_x,\qquad 0\preceq\bar\Lambda_x\preceq\Id,
 \qquad 0\preceq N_x\preceq p\Id,
 \qquad \sum_{i=1}^m w_{x,i}=d.
 \label{eq:basicLewis}
\end{equation}

\begin{lemma}[Analytic Lewis weights in logarithmic coordinates]
\label{lem:LewisAnalytic}
The map $x\mapsto \ell_x:=\log w_x$ is real analytic on
$\operatorname{int}K$. More precisely, for
$\ell\in\R^m$ set
\[
 W(\ell)=\Diag(e^{\ell_1},\ldots,e^{\ell_m}),\qquad
 B(x,\ell)=W(\ell)^\beta A_x
\]
and
\[
 P(x,\ell)=B(x,\ell)
 \bigl(B(x,\ell)^\transpose B(x,\ell)\bigr)^{-1}
 B(x,\ell)^\transpose.
\]
In a neighborhood of $(x,\ell_x)$, define
\begin{equation}
 \mathcal L(x,\ell)
 :=\ell-\log\diag P(x,\ell).
 \label{eq:LewisImplicitMap}
\end{equation}
Then $\mathcal L(x,\ell_x)=0$, and, at a solution,
\begin{align}
 D_\ell\mathcal L(x,\ell_x)
 &=\Id-c_pW_x^{-1}\Lambda_x\notag\\
 &=W_x^{-1/2}
 \bigl(\Id-c_p\bar\Lambda_x\bigr)W_x^{1/2}.
 \label{eq:LewisJacobian}
\end{align}
In particular, this Jacobian is invertible, with
\begin{equation}
 \bigl(D_\ell\mathcal L(x,\ell_x)\bigr)^{-1}
 =W_x^{-1/2}
 \bigl(\Id-c_p\bar\Lambda_x\bigr)^{-1}W_x^{1/2}.
 \label{eq:LewisJacobianInverse}
\end{equation}
\end{lemma}

\begin{proof}
Every row of $A_x$ is nonzero, so every diagonal entry of
$P(x,\ell)$ is positive. Thus \eqref{eq:LewisImplicitMap} is real
analytic wherever $x\in\operatorname{int}K$ and $\ell\in\R^m$.
We compute its weight Jacobian. If a full-column-rank matrix $B$
varies according to $\dot B=\Diag(z)B$, then its orthogonal
projection $P=B(B^\transpose B)^{-1}B^\transpose$ satisfies
\begin{equation}
 \dot P=(\Id-P)\Diag(z)P+P\Diag(z)(\Id-P).
 \label{eq:projectionVariation}
\end{equation}
Taking diagonal entries gives
\begin{equation}
 \diag\dot P=2\bigl(W-P^{\had2}\bigr)z=2\Lambda z,
 \qquad W=\Diag(\diag P).
 \label{eq:projectionDiagonalVariation}
\end{equation}
For a variation $\dot\ell=\zeta$ at a Lewis solution,
$\dot B=\beta\Diag(\zeta)B$. Hence
\[
 D_\ell\log\diag P(x,\ell_x)[\zeta]
 =2\beta W_x^{-1}\Lambda_x\zeta
 =c_pW_x^{-1}\Lambda_x\zeta,
\]
which proves \eqref{eq:LewisJacobian}. Since
$0\preceq\bar\Lambda_x\preceq\Id$ and $1-c_p=2/p>0$,
$\Id-c_p\bar\Lambda_x$ is positive definite. This proves
\eqref{eq:LewisJacobianInverse}.

The real analytic implicit-function theorem \cite{KrantzParks02} now
produces a unique analytic solution branch near every
$x\in\operatorname{int}K$. The positive
Lewis weights are unique, so any two local branches agree on their
overlap. These branches therefore patch to one real analytic map on
$\operatorname{int}K$.
\end{proof}

For $r\in\R^d$, write
$s_{x,r}=A_xr$ and $v_{x,r}=W_x^{1/2}s_{x,r}$. Differentiating
\eqref{eq:LewisImplicitMap} also recovers the Lewis derivative
identity. Indeed, at fixed $\ell$ one has
$D_xB[r]=-\Diag(s_{x,r})B$, so
\eqref{eq:projectionDiagonalVariation} gives
\[
 \bigl(\Id-c_pW_x^{-1}\Lambda_x\bigr)D\ell_x[r]
 =-2W_x^{-1}\Lambda_xs_{x,r}.
\]
Using \eqref{eq:LewisJacobianInverse}, the commutativity of
$\bar\Lambda_x$ and
$(\Id-c_p\bar\Lambda_x)^{-1}$, and \eqref{eq:Ndef}, we obtain
\begin{align}
 D\log w_x[r]
 &=-W_x^{-1/2}N_xv_{x,r},\notag\\
 D W_x[r]
 &=-\Diag\!\left(W_x^{1/2}N_xv_{x,r}\right).
 \label{eq:LewisDerivative}
\end{align}
Existence, uniqueness, strict positivity and the first-derivative
identity are part of the Lewis-weight calculus in
\cite[Section~5, in particular Lemma~24]{LS19}; the notation and the
bounds in \eqref{eq:basicLewis} are recorded explicitly in
\cite[Section~2.2]{Kook26}.

\subsection{Abstract local overlap and mixing reduction}

For reference, a $C^2$ metric $g$ is strongly self-concordant (SSC)
if, for every $x\in\operatorname{int}K$ and $h\in\R^d$,
\[
 \norm{g(x)^{-1/2}Dg(x)[h]g(x)^{-1/2}}_{\mathrm F}
 \le 2\norm{h}_{g(x)},
\]
and lower-trace self-concordant (LTSC) if, for the same $x$ and $h$,
\[
 \tr\bigl(g(x)^{-1}D^2g(x)[h,h]\bigr)\ge-\norm{h}_{g(x)}^2.
\]
The SSC inequality also implies the usual matrix self-concordance
inequality because the operator norm is bounded by the Frobenius
norm.
It is $\bar\nu$-symmetric when
\[
 D_g(x,1)\subseteq K\cap(2x-K)\subseteq D_g(x,\sqrt{\bar\nu})
\]
for every $x\in\operatorname{int}K$, where
$D_g(x,r)=\{y:\norm{y-x}_{g(x)}\le r\}$.
These sampling regularity notions follow \cite{LLV20,KV24}.

The next lemma is the point at which linearity of the target replaces
the function-counterpart assumption in the general Dikin framework.
It proves overlap through one common accepted mass and records every
interior and reverse-proposal event explicitly.

\begin{lemma}[One-step overlap for a linear exponential target]
\label{lem:linearTargetOverlap}
There are universal constants
$\eps_{\rm ov}\in(0,10^{-2})$ and $c_{\rm ov},c_{\rm near}>0$
with the following property. Let $g$ be strongly self-concordant and
lower-trace self-concordant, and suppose that
$R_{\eps_{\rm ov}}(g)\ge a$ for some $a>0$.  Put
\[
 0<\varrho\le c_{\rm ov}(1\wedge a).
\]
For every $\vartheta\in\R^d$ and $x\in\operatorname{int}K$, let
$\mathsf P_x$ be the one-step non-lazy Metropolis Dikin law started at
$x$, with proposal
$N(x,\varrho^2g(x)^{-1}/d)$ and target
$\pi_\vartheta\propto e^{-\vartheta^\transpose x}\mathbf1_K(x)$.
Then, for all $x,y\in\operatorname{int}K$,
\begin{equation}
 \norm{x-y}_{g(x)}
 \le \frac{c_{\rm near}\varrho}{\sqrt d}
 \quad\Longrightarrow\quad
 d_{\rm TV}(\mathsf P_x,\mathsf P_y)\le0.8 .
 \label{eq:linearTargetOverlap}
\end{equation}
The constants are independent of $\vartheta$, $d$, and the instance.
No function whose Hessian is comparable to $g$ is assumed.
\end{lemma}

\begin{proof}
We prove the overlap directly through the common accepted density.
Fix
\[
 \eps_{\rm ov}=2^{-10},\qquad
 \eta_0=2^{-8},\qquad t_0=\frac1{32}.
\]
Take $B=16$. Since $\E\norm{\xi}^2=d$, Markov's inequality gives,
for every $d\ge1$ and $\xi\sim N(0,\Id_d)$,
\begin{equation}
 \Pp\{\norm{\xi}>B\sqrt d\}\le\eta_0.
 \label{eq:GaussianNormChoice}
\end{equation}
We choose $c_{\rm ov}$ and $c_{\rm near}$ below; all restrictions on
them are universal.

For $s\in\{x,y\}$, let $q_s$ be the density of
\[
 N\!\left(s,\frac{\varrho^2}{d}g(s)^{-1}\right),
 \qquad u_s=z-s,
\]
and set $\phi=\frac12\log\det g$. For
$z\in\operatorname{int}K$, the logarithm of the Metropolis ratio from
$s$ is
\begin{align}
 \Lambda_s(z)
 &:={}
 \log\frac{\pi_\vartheta(z)q_z(s)}
 {\pi_\vartheta(s)q_s(z)}
 \notag\\
 &=-\vartheta^\transpose u_s+\phi(z)-\phi(s)
 -\frac{d}{2\varrho^2}
 \left(
 \norm{u_s}_{g(z)}^2-\norm{u_s}_{g(s)}^2
 \right).
 \label{eq:linearMetropolisLogRatio}
\end{align}

We first construct a good event for each starting point. The
restrictions below impose $c_{\rm ov}<1$, so
$\varrho<a$. The set of admissible ASC radii is downward closed;
hence Definition~\ref{def:qasc} applies at radius $\varrho$ and gives
an event
\[
 \mathcal A_s:=
 \left\{
 z\in\operatorname{int}K:\
 \abs{
 \norm{u_s}_{g(z)}^2-\norm{u_s}_{g(s)}^2
 }
 \le\frac{2\eps_{\rm ov}\varrho^2}{d}
 \right\}
\]
such that
\begin{equation}
 q_s(\mathcal A_s^c)\le\eps_{\rm ov}.
 \label{eq:ASCStartEvent}
\end{equation}
On $\mathcal A_s$, the final term in
\eqref{eq:linearMetropolisLogRatio} is at least $-\eps_{\rm ov}$.

We next control the determinant term. Matrix differentiation gives
\begin{align}
 D\phi(w)[v]
 &=\frac12\tr\!\left(g(w)^{-1}Dg(w)[v]\right),
 \label{eq:logdetFirst}\\
 D^2\phi(w)[v,v]
 &=\frac12\tr\!\left(g(w)^{-1}D^2g(w)[v,v]\right)
 \notag\\
 &\hspace{12mm}
 -\frac12\norm{
 g(w)^{-1/2}Dg(w)[v]g(w)^{-1/2}
 }_{\mathrm F}^2 .
 \label{eq:logdetSecond}
\end{align}
LTSC and SSC imply
\begin{equation}
 D^2\phi(w)[v,v]\ge-\frac52\norm{v}_{g(w)}^2,
 \qquad
 \norm{D\phi(w)}_{g(w)^{-1}}\le\sqrt d.
 \label{eq:logdetDifferentialBounds}
\end{equation}
Indeed, the first term in \eqref{eq:logdetSecond} is at least
$-\frac12\norm{v}_{g(w)}^2$, while SSC bounds the second one below by
$-2\norm{v}_{g(w)}^2$. The second inequality follows from
$\abs{\tr M}\le\sqrt d\norm{M}_{\mathrm F}$.

Define
\[
 \mathcal N_s:=\{\norm{u_s}_{g(s)}\le B\varrho\},
 \qquad
 \mathcal D_s:=\{D\phi(s)[u_s]\ge-t_0\}.
\]
Writing
$u_s=(\varrho/\sqrt d)g(s)^{-1/2}\xi$ and using
\eqref{eq:GaussianNormChoice} gives
\begin{equation}
 q_s(\mathcal N_s^c)\le\eta_0.
 \label{eq:GaussianNormStart}
\end{equation}
Moreover, $D\phi(s)[u_s]$ is centered Gaussian and, by
\eqref{eq:logdetDifferentialBounds}, has variance at most
$\varrho^2$. We may therefore choose $c_{\rm ov}$ sufficiently small
that
\begin{equation}
 \sup_{0<\varrho\le c_{\rm ov}}
 q_s(\mathcal D_s^c)\le\eta_0.
 \label{eq:GaussianDeterminantStart}
\end{equation}

On $\mathcal A_s\cap\mathcal N_s$, both endpoints of $[s,z]$ lie in
$\operatorname{int}K$, and convexity places the whole segment in the
interior. Put $\delta_s=\norm{u_s}_{g(s)}$. The matrix
self-concordance consequence needed here follows directly from SSC.
The case $u_s=0$ is immediate. Otherwise, if
$r_s(t)=\norm{u_s}_{g(s+tu_s)}$, then differentiation and the SSC
bound give
\[
 \abs{(r_s^2)'(t)}
 \le 2r_s(t)^3,
 \qquad\text{hence}\qquad
 \abs{r_s'(t)}\le r_s(t)^2.
\]
Since $r_s(0)=\delta_s$, comparison with the scalar differential
equation gives, for $0\le t\le1$,
\begin{equation}
 \norm{u_s}_{g(s+tu_s)}
 \le\frac{\delta_s}{1-t\delta_s}.
 \label{eq:pathMetricComparison}
\end{equation}
Impose
\[
 Bc_{\rm ov}\le\frac12,
 \qquad
 5B^2c_{\rm ov}^2\le\frac1{32}.
\]
Taylor's formula, \eqref{eq:logdetDifferentialBounds}, and
\eqref{eq:pathMetricComparison} yield
\begin{align}
 \phi(z)-\phi(s)-D\phi(s)[u_s]
 &=\int_0^1(1-t)
 D^2\phi(s+tu_s)[u_s,u_s]\dd t
 \notag\\
 &\ge
 -\frac52\delta_s^2
 \int_0^1\frac{1-t}{(1-t\delta_s)^2}\dd t
 \ge-5\delta_s^2.
 \label{eq:logdetTaylorLower}
\end{align}
Consequently, on
$\mathcal G_s:=\mathcal A_s\cap\mathcal N_s\cap\mathcal D_s$,
\begin{equation}
 \phi(z)-\phi(s)\ge-\frac1{16},
 \qquad
 q_s(\mathcal G_s^c)
 \le\eps_{\rm ov}+2\eta_0<\frac1{64}.
 \label{eq:startGoodEvent}
\end{equation}
Thus the two proposal-dependent terms in
\eqref{eq:linearMetropolisLogRatio} are at least $-1/8$ on
$\mathcal G_s$.

We also record the full comparison of the two forward Gaussian
densities. Set
\[
 \delta=\norm{x-y}_{g(x)}
 \le\frac{c_{\rm near}\varrho}{\sqrt d}.
\]
After decreasing $c_{\rm near}$, we have $\delta\le1/4$. Integrating
the SSC inequality along $[x,y]$ gives
\begin{equation}
 (1-\delta)^2g(x)
 \preceq g(y)
 \preceq(1-\delta)^{-2}g(x).
 \label{eq:nearbyMetricComparison}
\end{equation}
To see this without invoking a function counterpart, put
$h=y-x$ and $\gamma(t)=x+th$. The preceding scalar comparison gives
$\norm{h}_{g(\gamma(t))}\le\delta/(1-t\delta)$. For every fixed
$v\ne0$, SSC therefore implies
\[
 \abs{\frac{\dd}{\dd t}
 \log\bigl(v^\transpose g(\gamma(t))v\bigr)}
 \le\frac{2\delta}{1-t\delta}.
\]
Integration from $0$ to $1$ proves
\eqref{eq:nearbyMetricComparison}.
Let $M=g(x)^{-1/2}g(y)g(x)^{-1/2}$. The Gaussian relative entropy
formula \cite{CoverThomas06} is
\begin{equation}
 D_{\rm KL}(q_x\Vert q_y)
 =\frac12\left\{
 \tr M-d-\log\det M
 +\frac{d}{\varrho^2}\norm{x-y}_{g(y)}^2
 \right\}.
 \label{eq:nearbyGaussianKL}
\end{equation}
Every eigenvalue $\lambda$ of $M$ lies in
$[(1-\delta)^2,(1-\delta)^{-2}]$, so
$\lambda-1-\log\lambda\le C\delta^2$. Hence
\[
 D_{\rm KL}(q_x\Vert q_y)
 \le C\left(d\delta^2+\frac{d\delta^2}{\varrho^2}\right)
 \le Cc_{\rm near}^2.
\]
Pinsker's inequality \cite[Lemma~2.5]{Tsybakov09} gives
\begin{equation}
 d_{\rm TV}(q_x,q_y)\le Cc_{\rm near}.
 \label{eq:nearbyGaussianTV}
\end{equation}
Choose $c_{\rm near}$ so that the right side is at most $1/64$.
This calculation controls both the covariance and displaced-mean
quadratic terms.

It remains to pass from proposal overlap to accepted-density overlap.
For $s\in\{x,y\}$, define the accepted density without evaluating
$g(z)$ outside its domain by
\begin{align}
 k_s(z)
 &:={}
 \begin{cases}
 q_s(z)\bigl(1\wedge e^{\Lambda_s(z)}\bigr),
 &z\in\operatorname{int}K,\\
 0,&z\notin\operatorname{int}K,
 \end{cases}
 \notag\\
 &=
 \begin{cases}
 \min\left\{
 q_s(z),e^{-\vartheta^\transpose(z-s)}q_z(s)
 \right\},
 &z\in\operatorname{int}K,\\
 0,&z\notin\operatorname{int}K.
 \end{cases}
 \label{eq:acceptedDensity}
\end{align}
Thus
\begin{equation}
 \mathsf P_s
 =\left(1-\int_{\R^d}k_s(z)\dd z\right)\delta_{\{s\}}
 +k_s(z)\dd z.
 \label{eq:kernelAcceptedDecomposition}
\end{equation}
Here $\delta_{\{s\}}$ is the unit point mass at $s$.

Choose $s_\star\in\{x,y\}$ so that
\[
 \vartheta^\transpose s_\star
 =\min\{\vartheta^\transpose x,\vartheta^\transpose y\},
\]
and define the common halfspace
\[
 H:=\{z:\vartheta^\transpose z
 \le\vartheta^\transpose s_\star\}.
\]
For every $z\in H$ and $s\in\{x,y\}$,
$-\vartheta^\transpose(z-s)\ge0$. If $\vartheta\ne0$, symmetry of
the centered Gaussian proposal gives $q_{s_\star}(H)=1/2$; if
$\vartheta=0$, then $H=\R^d$. Thus
\begin{equation}
 q_{s_\star}(H)\ge\frac12.
 \label{eq:commonTargetHalfspace}
\end{equation}
On $H\cap\mathcal G_s$, all quantities in
\eqref{eq:linearMetropolisLogRatio} are defined,
$\Lambda_s(z)\ge-1/8$, and therefore
\begin{equation}
 k_s(z)\ge e^{-1/8}q_s(z).
 \label{eq:acceptedDensityLower}
\end{equation}

Let $m(z)=\min\{q_x(z),q_y(z)\}$. Since $m\le q_s$,
\begin{align}
 \int_{\R^d}\min\{k_x(z),k_y(z)\}\dd z
 &\ge e^{-1/8}
 \int_{H\cap\mathcal G_x\cap\mathcal G_y}m(z)\dd z
 \notag\\
 &\ge e^{-1/8}\left(
 \int_Hm(z)\dd z
 -q_x(\mathcal G_x^c)-q_y(\mathcal G_y^c)
 \right).
 \label{eq:commonAcceptedMassFirst}
\end{align}
Moreover,
\begin{align}
 \int_Hm(z)\dd z
 &\ge q_{s_\star}(H)
 -\int_{\R^d}\bigl(q_{s_\star}(z)-m(z)\bigr)\dd z
 \notag\\
 &=q_{s_\star}(H)-d_{\rm TV}(q_x,q_y)
 \ge\frac12-\frac1{64}.
 \label{eq:commonProposalHalfspace}
\end{align}
Equations \eqref{eq:startGoodEvent},
\eqref{eq:commonAcceptedMassFirst}, and
\eqref{eq:commonProposalHalfspace} give
\[
 \int_{\R^d}\min\{k_x(z),k_y(z)\}\dd z
 \ge e^{-1/8}\left(\frac12-\frac3{64}\right)
 >\frac15.
\]

If $x=y$, there is nothing to prove. If $x\ne y$, the rejection atoms
in \eqref{eq:kernelAcceptedDecomposition} are supported at distinct
points, while the accepted parts are absolutely continuous.
Therefore their common mass is exactly the integral above, and
\[
 d_{\rm TV}(\mathsf P_x,\mathsf P_y)
 =1-\int_{\R^d}\min\{k_x(z),k_y(z)\}\dd z
 <\frac45.
\]
This proves \eqref{eq:linearTargetOverlap}. The construction used
only SSC, LTSC, and the ASC event, including its interior condition;
no function counterpart is needed.
\end{proof}

Having obtained common accepted mass, it remains only to convert this
local overlap into global conductance through the symmetry geometry.

\begin{lemma}[Symmetry, cross-ratio separation, and conductance]
\label{lem:symmetryConductance}
Let $\pi$ be a log-concave probability distribution on a bounded
full-dimensional convex body $K$, and let $g$ be
$\bar\nu$-symmetric.  Suppose a reversible Markov kernel
$\mathsf K$ with stationary law $\pi$ satisfies, for some
$\Delta\in(0,1)$ and all $x,y\in\operatorname{int}K$,
\begin{equation}
 \norm{x-y}_{g(x)}\le\Delta
 \quad\Longrightarrow\quad
 d_{\rm TV}(\mathsf K_x,\mathsf K_y)\le0.9.
 \label{eq:localOverlapAbstract}
\end{equation}
Then its global stationary conductance satisfies
\begin{equation}
 \Phi(\mathsf K)\ge
 \frac{c\Delta}{\sqrt{\bar\nu}}
 \label{eq:symmetryConductance}
\end{equation}
for a universal constant $c>0$.
\end{lemma}

\begin{proof}
Symmetry yields the required metric isoperimetry. For distinct
$x,y\in\operatorname{int}K$, let the
line through them meet $\partial K$ at $p,q$, ordered as
$p,x,y,q$, and let $d_K(x,y)$ denote the cross-ratio distance
\[
 d_K(x,y)
 =\frac{\norm{x-y}_2\norm{p-q}_2}
 {\norm{p-x}_2\norm{y-q}_2}.
\]
Let $u$ be the unit vector along this line and let
$t_x$ be the smaller of the two Euclidean distances from $x$ to the
endpoints.  Both $x+t_xu$ and $x-t_xu$ lie in $K$, so
$t_xu\in K\cap(2x-K)-x$.  The outer symmetry inclusion gives
$t_x\norm u_{g(x)}\le\sqrt{\bar\nu}$.  Directly from the displayed
cross-ratio formula,
\begin{equation}
 d_K(x,y)\ge
 \frac{\norm{x-y}_2}{t_x}
 \ge\frac{\norm{x-y}_{g(x)}}{\sqrt{\bar\nu}}.
 \label{eq:crossRatioLocalNorm}
\end{equation}
The first inequality holds whichever of the two endpoint distances
defines $t_x$.

The cross-ratio isoperimetric inequality for log-concave measures
\cite[Theorem~2.5]{LovaszVempala07}, in the form used in
\cite[proof of Lemma~B.6]{KV24}, says that if
$K=S_1\sqcup S_2\sqcup S_3$, then
\begin{equation}
 \pi(S_3)\ge
 c\,d_K(S_1,S_2)\pi(S_1)\pi(S_2).
 \label{eq:crossRatioIsoperimetry}
\end{equation}
Equations \eqref{eq:crossRatioLocalNorm} and
\eqref{eq:crossRatioIsoperimetry} give the required global
isoperimetry in the local norm of $g$.

For the conductance partition, in the standard local-overlap form of
\cite{LovaszSimonovits93,KV24}, fix $S$ with
$0<\pi(S)\le1/2$ and choose a fixed $\alpha<0.05$. Put
\begin{align*}
 Q(S,S^c)&=\int_S\mathsf K(x,S^c)\pi(\dd x),\\
 S_1&=\{x\in S:\mathsf K(x,S^c)<\alpha\},\\
 S_2&=\{y\in S^c:\mathsf K(y,S)<\alpha\}.
\end{align*}
For $x\in S_1$ and $y\in S_2$,
$d_{\rm TV}(\mathsf K_x,\mathsf K_y)>1-2\alpha>0.9$.
Hence \eqref{eq:localOverlapAbstract} implies
$\norm{x-y}_{g(x)}>\Delta$, and
\eqref{eq:crossRatioLocalNorm} separates $S_1,S_2$ by at least
$\Delta/\sqrt{\bar\nu}$ in cross ratio.  If either $S_1$ or $S_2$
has lost a fixed fraction of its side, the definitions and
reversibility already give
$Q(S,S^c)\ge c\pi(S)$.  Otherwise both retain a fixed fraction, and
\eqref{eq:crossRatioIsoperimetry} gives
\[
 \pi\bigl(K\setminus(S_1\cup S_2)\bigr)
 \ge \frac{c\Delta}{\sqrt{\bar\nu}}\pi(S).
\]
The complement is contained in
$(S\setminus S_1)\cup(S^c\setminus S_2)$, whose stationary mass is
at most $2Q(S,S^c)/\alpha$.  Thus
$Q(S,S^c)\ge c\Delta\pi(S)/\sqrt{\bar\nu}$.  Taking the infimum over
$S$ proves \eqref{eq:symmetryConductance}.  This is the local-overlap
partition of \cite[Lemma~B.2]{KV24}, written here with its
cross-ratio separation explicit.
\end{proof}

\begin{proposition}[Chi-square corollary of the Dikin mixing framework]
\label{prop:framework}
Let $K\subset\R^d$ be bounded, convex and full-dimensional, and let
$g:\operatorname{int}K\to\mathbb S_{++}^d$ be a $C^2$ metric that is
strongly self-concordant, lower-trace self-concordant and
$\bar\nu$-symmetric.  Suppose that
$R_\eps(g)\ge\mathfrak a(\eps)>0$ for every $\eps\in(0,1)$, where
$\mathfrak a$ is independent of the dimension and of the instance.
There are universal constants $\eps_{\rm mix}\in(0,10^{-2})$ and
$c_0\in(0,1)$ and $c_1,c_2>0$ with the following property.  If
\begin{equation}
 a_0:=1\wedge\mathfrak a(\eps_{\rm mix}),
 \qquad
 \varrho_{\rm KV}:=c_0a_0,
 \label{eq:frameworkRadius}
\end{equation}
then, for every $\vartheta\in\R^d$, the $1/2$-lazy Metropolis Dikin
kernel $\mathsf K$ with proposal radius $\varrho_{\rm KV}$ is
reversible with stationary law $\pi_\vartheta$, has global
conductance
\begin{equation}
 \Phi(\mathsf K)
 :=\inf_{0<\pi_\vartheta(S)\le1/2}
 \frac{\int_S\mathsf K(x,S^c)\pi_\vartheta(\dd x)}
 {\pi_\vartheta(S)}
 \ge \frac{c_1a_0}{\sqrt{d\bar\nu}},
 \label{eq:frameworkConductance}
\end{equation}
and has $L^2(\pi_\vartheta)$ spectral gap at least
$c_2a_0^2/(d\bar\nu)$.  Consequently, for every
$\mu_0\ll\pi_\vartheta$ with finite chi-square divergence and every
integer $T\ge0$,
\begin{equation}
 \chi^2(\mu_0\mathsf K^T\Vert\pi_\vartheta)
 \le
 \exp\!\left(-\frac{c_2a_0^2T}{d\bar\nu}\right)
 \chi^2(\mu_0\Vert\pi_\vartheta).
 \label{eq:frameworkContraction}
\end{equation}
Consequently, for every $\delta\in(0,1)$, the left side is at most
$\delta$ whenever
\begin{equation}
 T\ge C a_0^{-2}d\bar\nu
 \log\!\left(
 \frac{1+\chi^2(\mu_0\Vert\pi_\vartheta)}{\delta}
 \right).
 \label{eq:framework}
\end{equation}
\end{proposition}

\begin{proof}
We derive the chi-square statement from the total-variation framework
of \cite{KV24}. Take $\eps_{\rm mix}=\eps_{\rm ov}$ from
Lemma~\ref{lem:linearTargetOverlap}.  Since
\[
 \varrho_{\rm KV}=c_0a_0,\qquad
 a_0\le1\wedge R_{\eps_{\rm mix}}(g),
\]
choosing $c_0\le c_{\rm ov}$ lets that lemma apply directly.  Thus
the non-lazy kernels satisfy
\eqref{eq:linearTargetOverlap}, uniformly in $\vartheta$. For the lazy
kernels
$\mathsf K_x=(\delta_x+\mathsf P_x)/2$, the triangle inequality gives
\[
 d_{\rm TV}(\mathsf K_x,\mathsf K_y)
 \le\frac12+\frac12(0.8)=0.9.
\]
Lemma~\ref{lem:symmetryConductance}, with
$\Delta=c_{\rm near}\varrho_{\rm KV}/\sqrt d$, now gives
\[
 \Phi(\mathsf K)
 \ge\frac{c\varrho_{\rm KV}}{\sqrt{d\bar\nu}}
 =\frac{ca_0}{\sqrt{d\bar\nu}}.
\]
This is the global stationary conductance in
\eqref{eq:frameworkConductance}, not a warm or restricted
conductance, and proves that display after adjusting $c_1$.

\emph{From conductance to chi-square contraction.}
Metropolis detailed balance makes $\mathsf K$ self-adjoint on
$L^2(\pi_\vartheta)$.  Since
$\mathsf K=(\Id+\mathsf P)/2$, its spectrum lies in $[0,1]$.
Cheeger's inequality for reversible Markov kernels
\cite[Theorem~2.1]{LawlerSokal88} therefore yields
\[
 \operatorname{gap}(\mathsf K)
 \ge\frac{\Phi(\mathsf K)^2}{2}
 \ge\frac{c a_0^2}{d\bar\nu}.
\]
If $f_0=\dd\mu_0/\dd\pi_\vartheta-1$, detailed balance gives
\[
 \frac{\dd(\mu_0\mathsf K^T)}{\dd\pi_\vartheta}-1
 =\mathsf K^Tf_0.
\]
The function $f_0$ has $\pi_\vartheta$-mean zero, and hence
\begin{align*}
 \chi^2(\mu_0\mathsf K^T\Vert\pi_\vartheta)
 &=\norm{\mathsf K^Tf_0}_{L^2(\pi_\vartheta)}^2\\
 &\le
 \left(1-\frac{ca_0^2}{d\bar\nu}\right)^{2T}
 \norm{f_0}_{L^2(\pi_\vartheta)}^2\\
 &\le
 \exp\!\left(-\frac{c_2a_0^2T}{d\bar\nu}\right)
 \chi^2(\mu_0\Vert\pi_\vartheta).
\end{align*}
No warmness bound is used: finite chi-square divergence is needed
only for $f_0\in L^2(\pi_\vartheta)$.  Solving the last display for
$T$ gives \eqref{eq:framework}; replacing the numerator
$\chi^2(\mu_0\Vert\pi_\vartheta)$ by
$1+\chi^2(\mu_0\Vert\pi_\vartheta)$ also covers the stationary
initialization without a separate case.  The same chi-square
formulation is recorded in \cite[Lemma~2.2]{Kook26}.
\end{proof}

The mixing argument is now reduced to deterministic Lee--Sidford
regularity and the quantitative ASC estimate of
Theorem~\ref{thm:asc}.

\subsection{Deterministic Lee--Sidford geometry and scaling}

\begin{proposition}[Established Lee--Sidford regularity]
\label{prop:knownregularity}
Let $p\ge\max\{4,2(1+\log m)\}$ and set
\begin{equation}
 \lambda_p:=(1+p^2)(1+p),\qquad H_p:=\lambda_pg_0.
 \label{eq:LLVnormalization}
\end{equation}
Then $H_p$ is strongly self-concordant and lower-trace
self-concordant, and it is $\bar\nu_p$-symmetric with
\begin{equation}
 \bar\nu_p\le e\lambda_pd.
 \label{eq:LSsymmetry}
\end{equation}
\end{proposition}

\begin{proof}
Under the dictionary $q=p$ and
$A_x^\transpose W_x^{1-2/q}A_x=g_0(x)$,
\eqref{eq:LLVnormalization} is exactly the LS matrix of
\cite[Definition~6]{LLV20}. Strong self-concordance and convexity of
$\log\det H_p$ are \cite[Lemmas~4.2 and~4.3]{LLV20}. The latter gives
\[
 \tr(H_p^{-1}D^2H_p[h,h])
 \ge
 \tr(H_p^{-1}DH_p[h]H_p^{-1}DH_p[h])\ge0,
\]
hence lower-trace self-concordance. The proof of
\cite[Lemma~4.3]{LLV20} also gives
\[
 \norm{y-x}_{H_p(x)}^2
 \le\lambda_p\sum_iw_{x,i}^{1-2/p}
 \le\lambda_pd^{1-2/p}m^{2/p}
 \le e\lambda_pd
\]
for $y\in K\cap(2x-K)$; here
$p\ge2(1+\log m)$ implies $m^{2/p}\le e$. The reverse inclusion
follows directly from the leverage identity
\[
 a_{x,i}^\transpose g_0(x)^{-1}a_{x,i}=w_{x,i}^{2/p}.
\]
Here $a_{x,i}^\transpose$ denotes the $i$th row of $A_x$.
Since $w_{x,i}=(P_x)_{ii}\le1$, if
$\norm{y-x}_{H_p(x)}\le1$, Cauchy--Schwarz gives
\[
 \abs{a_{x,i}^\transpose(y-x)}
 \le\lambda_p^{-1/2}w_{x,i}^{1/p}\le1
\]
for every $i$, and hence $y\in K\cap(2x-K)$. This proves
\eqref{eq:LSsymmetry}. See also
\cite[Section~2.2]{Kook26} for this normalization.
\end{proof}

\begin{proposition}[Metric scaling]\label{prop:scaling}
For every $L>0$,
\begin{equation}
 R_\eps(Lg)=L^{1/2}R_\eps(g).
 \label{eq:radiusScaling}
\end{equation}
If $L\ge1$ and $g$ is $\bar\nu$-symmetric, then $Lg$ is $L\bar\nu$-symmetric. Scaling by $L\ge1$ also preserves strong self-concordance and lower-trace self-concordance.
\end{proposition}

\begin{proof}
A radius-$r$ proposal for $Lg$ is a radius-$r/L^{1/2}$
proposal for $g$. Its discrepancy is multiplied by $L$, while
\[
 L\frac{2\eps(r/L^{1/2})^2}{d}
 =\frac{2\eps r^2}{d},
\]
which proves \eqref{eq:radiusScaling}. The symmetry statement follows
from $\norm{u}_{Lg(x)}^2=L\norm{u}_{g(x)}^2$. In the defining
differential inequalities for strong self-concordance, the left side
is unchanged while the local norm on the right is multiplied by
$L^{1/2}$. For lower-trace self-concordance, the trace term is
unchanged while the negative lower bound is multiplied by $L$; hence
the condition is preserved when $L\ge1$.
\end{proof}

\subsection{Affine covariance}

\begin{lemma}[Affine covariance]\label{lem:affineCovariance}
Let $\Psi(\widetilde x)=x_c+T\widetilde x$, where
$T\in\R^{d\times d}$ is invertible, and pull the polytope and metric
back through $\Psi$. Then the Lewis weights and the matrices $P$ and
$N$ are unchanged at corresponding points, while
\begin{equation}
 \widetilde g_0(\widetilde x)
 =T^\transpose g_0(\Psi(\widetilde x))T.
 \label{eq:metricCovariance}
\end{equation}
Moreover,
\begin{align}
 D^k\widetilde N_{\widetilde x}
 [r_1,\ldots,r_k]
 &=D^kN_{\Psi(\widetilde x)}
 [T r_1,\ldots,T r_k],
 \label{eq:NjetCovariance}\\
 \norm{r}_{\widetilde g_0(\widetilde x)}
 &=\norm{T r}_{g_0(\Psi(\widetilde x))}.
 \label{eq:localNormCovariance}
\end{align}
The Gaussian proposal and the discrepancy in
\eqref{eq:quantASCdef} are carried into their counterparts by
$\Psi$.
\end{lemma}

\begin{proof}
The pulled-back inequality matrix is $\widetilde A=AT$, and the slack
vectors at $\widetilde x$ and $\Psi(\widetilde x)$ coincide. Hence
\[
 \widetilde A_{\widetilde x}=A_{\Psi(\widetilde x)}T.
\]
Substitution in \eqref{eq:lewis} cancels $T$ and $T^\transpose$
inside the inverse, so the Lewis fixed point, and therefore $P$ and
$N$, is unchanged. The same substitution in \eqref{eq:g0} proves
\eqref{eq:metricCovariance}. The chain rule gives
\eqref{eq:NjetCovariance}, and \eqref{eq:localNormCovariance} follows
by congruence. Finally,
\[
 T\widetilde g_0(\widetilde x)^{-1}T^\transpose
 =g_0(\Psi(\widetilde x))^{-1},
\]
so corresponding Gaussian proposals have the same law after applying
$\Psi$. The two squared local norms in
\eqref{eq:quantASCdef} are preserved by
\eqref{eq:metricCovariance}.
\end{proof}

The covariance lemma allows every ASC estimate to be proved at an
arbitrary base point after imposing $x=0$ and $g_0(x)=\Id_d$.

\section{Selective expansion of the metric discrepancy}
\label{sec:selective}

The normalization, path variables, and first bottleneck decomposition
in this section follow \cite[Sections~1.2 and~3.1]{Kook26}. We
retain that notation to make the comparison transparent. Kook
iterates the recursion to fixed order. Lemma~\ref{lem:selective-expansion}
records its exact arbitrary-order form, which will later be used with
a logarithmic truncation order.

By Lemma~\ref{lem:affineCovariance}, the metric transforms
covariantly, while the ASC discrepancy and radius are affine
invariant. Fix a base point $x$ and apply an affine change of
coordinates so that
\begin{equation}
 x=0,
 \qquad g_0(x)=\Id_d.
 \label{eq:normalization}
\end{equation}
Let $h\sim N(0,\Id_d)$, set $\eta=r/\sqrt d$, and define
\begin{equation}
 F(t)=h^\transpose g_0(x+th)h,
 \qquad 0\le t\le\eta.
 \label{eq:Fdef}
\end{equation}
For $x_t=x+th$, abbreviate $A_t=A_{x_t}$, $W_t=W_{x_t}$ and $N_t=N_{x_t}$, and set
\begin{equation}
 s_t=A_th,
 \quad u_t=W_t^\beta s_t,
 \quad v_t=W_t^{1/2}s_t,
 \quad \zeta_t=W_t^{-1/p}s_t,
 \quad q_t=W_t^{-1/2}u_t^{\had2}.
 \label{eq:pathvars}
\end{equation}
Then $F(t)=\norm{u_t}^2$. Since $s_t'=-s_t^{\had2}$ and \eqref{eq:LewisDerivative} holds, if $\ell_t=\log w_t$, then
\begin{equation}
 F'(t)=-2E_1(t)-2\beta H_1(t),
 \qquad
 E_1(t)=\ip{s_t}{W_t^\alpha s_t^{\had2}},
 \qquad
 H_1(t)=q_t^\transpose N_tv_t.
 \label{eq:Fprime}
\end{equation}
For $j\ge1$, define
\begin{equation}
 H_j(t)=q_t^\transpose N_t^{(j-1)}v_t,
 \label{eq:Hj}
\end{equation}
where $N_t^{(k)}=D^kN_{x_t}[h,\ldots,h]$. Differentiation gives
\begin{equation}
 H_j'(t)=G_{j+1}(t)+H_{j+1}(t),
 \quad
 G_{j+1}(t)=(q_t')^\transpose N_t^{(j-1)}v_t
 +q_t^\transpose N_t^{(j-1)}v_t'.
 \label{eq:selective}
\end{equation}

\begin{lemma}[Selective expansion]\label{lem:selective-expansion}
For every integer $J\ge1$,
\begin{align}
 \int_0^\eta H_1(t)\dd t
 ={}&\sum_{j=1}^J\frac{\eta^j}{j!}H_j(0)
 +\sum_{j=2}^J\int_0^\eta
 \frac{(\eta-t)^{j-1}}{(j-1)!}G_j(t)\dd t \notag\\
 &+\int_0^\eta\frac{(\eta-t)^J}{J!}
 \bigl(G_{J+1}(t)+H_{J+1}(t)\bigr)\dd t.
 \label{eq:selective-exact}
\end{align}
\end{lemma}

\begin{proof}
Iterate
\[
 H_j(t)=H_j(0)+\int_0^t\bigl(G_{j+1}(s)+H_{j+1}(s)\bigr)\dd s
\]
and apply Fubini's theorem. After $j$ integrations the kernel is $(\eta-t)^{j-1}/(j-1)!$.
\end{proof}

The expansion separates the proof into three tasks. Analytic frame
estimates control the derivatives of $N_x$; the weighted Wick theorem
controls the base-point polynomials $H_j(0)$; and pathwise localization
controls the $G_j(t)$ terms and the terminal remainder. The next four
sections treat these tasks in that order.

\section{Analytic horizontal-frame jets}\label{sec:analytic}

Following \cite[Section~1.2 and Appendix~B.3]{Kook26}, we express
Lewis-weight derivatives through moving orthonormal frames. A
horizontal gauge and a complex relative-row neighborhood yield
dimension-free Cauchy bounds for all derivatives of $N_x$, uniformly
in $m,d$ and the smallest Lewis weight.

\subsection{Horizontal-frame coordinates}

\begin{lemma}[Algebraic identities for a physical frame]
\label{lem:physicalFrameAlgebra}
Fix $y\in\operatorname{int}K$, put
$\widehat B_y=W_y^\beta A_y$, and let $U$ be any real orthonormal
frame for $\operatorname{col}(\widehat B_y)$. Then
\begin{align}
 UU^\transpose
 &=\widehat B_y
 (\widehat B_y^\transpose\widehat B_y)^{-1}
 \widehat B_y^\transpose=P_y,
 \label{eq:physicalProjection}\\
 \diag(UU^\transpose)&=w_y.
 \label{eq:physicalWeights}
\end{align}
In particular, every row of $U$ is nonzero.
\end{lemma}

\begin{proof}
Two full-column-rank matrices with the same column space have the same
orthogonal projection. This gives \eqref{eq:physicalProjection}; the
Lewis fixed-point equation \eqref{eq:lewis} then gives
\eqref{eq:physicalWeights}. Strict positivity of every Lewis weight
and \eqref{eq:physicalWeights} imply that no row of $U$ vanishes.
\end{proof}

Fix a real base point satisfying \eqref{eq:normalization} and a direction $e\in\R^d$ with $\norm{e}=1$. Along $y_\tau=x+\tau e$, let
\[
 \widehat B_\tau=W_{y_\tau}^\beta A_{y_\tau}.
\]
Lemma~\ref{lem:LewisAnalytic} makes this a real analytic, full-rank
matrix path. Orthonormalize it and apply the horizontal-gauge
construction of Lemma~\ref{lem:horizontalgauge}. We thereby obtain
the physical horizontal frame $U_\tau$, choosing its initial gauge so
that
\begin{equation}
 U_0=\widehat B_0=W_x^\beta A_x.
 \label{eq:physicalU0}
\end{equation}
This is possible because
$\widehat B_0^\transpose\widehat B_0=g_0(x)=\Id_d$. The frame satisfies
\begin{equation}
 U_\tau^\transpose U_\tau=\Id_d,
 \qquad U_\tau^\transpose U_\tau'=0.
 \label{eq:horizontal}
\end{equation}
The horizontal gauge and its evolution equation are proved in Appendix~\ref{app:frame}.

Write $u_i^\transpose$ for the $i$th row of $U$ and define, initially on the real path,
\begin{equation}
 w_i=u_i^\transpose u_i,
 \qquad \omega_i=w_i^{1/2},
 \qquad \phi_i=\frac{u_i\otimes u_i}{\omega_i}.
 \label{eq:rowmap}
\end{equation}
Let $\Phi$ be the matrix whose $i$th row is $\phi_i^\transpose$, after identifying $\R^d\otimes\R^d$ with $\R^{d^2}$. Then
\begin{equation}
 G:=\Phi\Phi^\transpose
 =W^{-1/2}P^{\had2}W^{-1/2},
 \qquad P=UU^\transpose,
 \label{eq:Gframe}
\end{equation}
and, with
\[
 a=\frac2p,\qquad c_p=1-\frac2p,
\]
we have
\begin{equation}
 N(U)=2(\Id-G)(a\Id+c_pG)^{-1}.
 \label{eq:Nframe}
\end{equation}
Indeed,
$a\Id+c_pG=\Id-c_p(\Id-G)=\Id-c_p\bar\Lambda$.
For the physical frame just constructed,
Lemma~\ref{lem:physicalFrameAlgebra} and the row definitions give
\begin{equation}
 G(U_\tau)
 =W_{y_\tau}^{-1/2}P_{y_\tau}^{\had2}W_{y_\tau}^{-1/2},
 \qquad N(U_\tau)=N_{y_\tau}.
 \label{eq:physicalGN}
\end{equation}

Let
\[
 s_i(\tau)=(A_{y_\tau}e)_i.
\]
Because each slack is affine,
\begin{equation}
 s_i(\tau)=\frac{s_i(0)}{1+\tau s_i(0)}.
 \label{eq:sexplicit}
\end{equation}
At the base point, $A_x=W_0^{-\beta}U_0$, and hence
\begin{equation}
 \abs{s_i(0)}
 =w_i^{-\beta}\abs{u_i^\transpose e}
 \le w_i^{1/2-\beta}=w_i^{1/p}\le1,
 \qquad
 \norm{W_0^{1/2}s(0)}\le1.
 \label{eq:sinitial}
\end{equation}
The second inequality follows from $W_0^{1/2}s(0)=W_0^{1/p}U_0e$ and $0<W_0\preceq\Id$.

Define
\begin{equation}
 v(U,\tau)=W(U)^{1/2}s(\tau),
 \qquad
 z(U,\tau)=-s(\tau)-\beta W(U)^{-1/2}N(U)v(U,\tau),
 \label{eq:zclosed}
\end{equation}
where $W(U)=\Diag(u_i^\transpose u_i)$. From \eqref{eq:LewisDerivative},
$\widehat B_\tau'=\Diag(z_\tau)\widehat B_\tau$. The horizontal condition therefore gives the closed equation
\begin{equation}
 U'=\cF(U,\tau)
 :=(\Id-UU^\transpose)\Diag(z(U,\tau))U,
 \qquad U(0)=U_0.
 \label{eq:frameODE}
\end{equation}
The converse identification of a solution of this closed equation
with the physical frame is proved after the analytic neighborhood has
been established.

\subsection{The weighted complex neighborhood}

We complexify all matrices and use algebraic transpose
$A^\transpose$, rather than the Hermitian adjoint
$A^*=\overline A^\transpose$, in the Stiefel equations and the
holomorphic maps $UU^\transpose$ and $\Phi\Phi^\transpose$. Norms are
the Hermitian Euclidean, operator and Frobenius norms. For these
norms,
\begin{equation}
 \norm{A^\transpose}_{\op}=\norm{A}_{\op},
 \qquad
 \norm{A^\transpose}_{\mathrm F}=\norm{A}_{\mathrm F};
 \label{eq:complexTransposeNorm}
\end{equation}
the first identity follows because the nonzero eigenvalues of
$\overline A A^\transpose$ and $A^\transpose\overline A$ agree, while
$A^\transpose\overline A=(A^*A)^\transpose$. Loewner-order statements
are used only at the real base point; away from it,
invertibility follows by operator-norm perturbation and a Neumann
series. Fix $W_0=\Diag(w_{0,i})$ and define
\begin{equation}
 \norm{V}_{\cX}
 :=\norm{V}_{\mathrm F}
 +\max_{1\le i\le m}\frac{\norm{e_i^\transpose V}_2}{\sqrt{w_{0,i}}}.
 \label{eq:Xnorm}
\end{equation}
For an $m\times q$ matrix $X$, set
\begin{equation}
 \norm{X}_{\mathrm{row},0}
 :=\max_i\frac{\norm{e_i^\transpose X}_2}{\sqrt{w_{0,i}}},
 \label{eq:rownorm}
\end{equation}
and, for $y\in\C^m$, define
\begin{equation}
 \norm{y}_{\mathrm{adm},0}=\norm{y}_2+\norm{W_0^{-1/2}y}_\infty.
 \label{eq:admnorm}
\end{equation}

\begin{lemma}[Analytic row map]\label{lem:rowmap}
There are universal constants $c_0,C>0$ such that, if
$\norm{U-U_0}_{\cX}\le\delta\le c_0$, then each $u_i^\transpose u_i$ lies in
$\{z:\abs{z-w_{0,i}}\le w_{0,i}/2\}$. The square-root branch agreeing with $\sqrt{w_{0,i}}$ is analytic, and
\begin{align}
 \norm{\omega(U)-\omega(U_0)}_2
 &\le C\norm{U-U_0}_{\mathrm F},
 \label{eq:omegaLip}\\
 \norm{\Phi(U)-\Phi(U_0)}_{\mathrm F}
 +\norm{\Phi(U)-\Phi(U_0)}_{\mathrm{row},0}
 &\le C\norm{U-U_0}_{\cX},
 \label{eq:PhiLip}\\
 \norm{G(U)-G(U_0)}_{\op}
 +\norm{G(U)-G(U_0)}_{\mathrm{row},0}
 &\le C\norm{U-U_0}_{\cX}.
 \label{eq:GLip}
\end{align}
The maps $U\mapsto\omega(U)$, $U\mapsto\Phi(U)$ and $U\mapsto G(U)$ are holomorphic on this neighborhood.
\end{lemma}

\begin{proof}
Write $u_i=u_{0,i}+\Delta_i$. By \eqref{eq:Xnorm},
$\norm{\Delta_i}\le\delta\sqrt{w_{0,i}}$, so
\[
 \abs{u_i^\transpose u_i-w_{0,i}}
 \le2\norm{u_{0,i}}\norm{\Delta_i}+\norm{\Delta_i}^2
 \le(2\delta+\delta^2)w_{0,i}.
\]
For sufficiently small $c_0$, this is at most $w_{0,i}/2$.

Consider
\[
 \cR(u)=\frac{u\otimes u}{(u^\transpose u)^{1/2}}.
\]
After writing $u=\sqrt{w_{0,i}}\,\widehat u$, all normalized rows lie
in one fixed complex neighborhood of the real unit sphere. On that
neighborhood the selected branch of
$\widehat u\mapsto(\widehat u^\transpose\widehat u)^{1/2}$ and the
normalized row map
$\widehat u\mapsto\widehat u\otimes\widehat u/
(\widehat u^\transpose\widehat u)^{1/2}$ are holomorphic with
universally bounded derivatives. Both unnormalized maps are
homogeneous of degree one. Thus integration along the segment from
$u_{0,i}$ to $u_i$ gives, with no dependence on $w_{0,i}$,
\[
 \abs{\omega_i(U)-\omega_i(U_0)}+\norm{\cR(u_i)-\cR(u_{0,i})}
 \le C\norm{\Delta_i}.
\]
Squaring and summing proves \eqref{eq:omegaLip} and the Frobenius part of \eqref{eq:PhiLip}; division by $\sqrt{w_{0,i}}$ proves the row part.

At the real base point, $G_0=\Phi_0\Phi_0^\transpose\preceq\Id$,
hence $\norm{\Phi_0}_{\op}\le1$. Equation \eqref{eq:PhiLip} gives
$\norm\Phi_{\op}\le1+C\delta$. From
\[
 G-G_0=(\Phi-\Phi_0)\Phi^\transpose+\Phi_0(\Phi-\Phi_0)^\transpose
\]
we obtain
\[
 \norm{G-G_0}_{\op}
 \le \norm{\Phi-\Phi_0}_{\op}
 \bigl(\norm{\Phi}_{\op}+\norm{\Phi_0}_{\op}\bigr)
 \le C\norm{U-U_0}_{\cX}.
\]
For the $i$th row, use
\[
 e_i^\transpose(G-G_0)
 =(\phi_i-\phi_{0,i})^\transpose\Phi^\transpose
 +\phi_{0,i}^\transpose(\Phi-\Phi_0)^\transpose
\]
and $\norm{\phi_{0,i}}=\sqrt{w_{0,i}}$. The already proved
Frobenius and relative-row bounds yield
\[
 \norm{e_i^\transpose(G-G_0)}_2
 \le C\sqrt{w_{0,i}}\norm{U-U_0}_{\cX},
\]
which proves the row part of \eqref{eq:GLip}. Holomorphy follows from
the selected square-root branches and matrix multiplication.
\end{proof}

\begin{lemma}[Resolvent stability]\label{lem:resolvent}
Let
\[
 R(U)=(a\Id+c_pG(U))^{-1}.
\]
There are universal $c_1,C>0$ such that, whenever
$\norm{U-U_0}_{\cX}\le\delta\le c_1a$,
\begin{equation}
 \norm{R(U)}_{\op}\le\frac2a.
 \label{eq:Rop}
\end{equation}
For every $m\times q$ matrix $X$,
\begin{equation}
 \norm{R(U)X}_{\mathrm{row},0}
 \le Ca^{-1}\norm{X}_{\mathrm{row},0}+Ca^{-2}\norm{X}_{\op}.
 \label{eq:Rrow}
\end{equation}
Furthermore,
\begin{equation}
 \norm{N(U)-N(U_0)}_{\op}
 +\norm{N(U)-N(U_0)}_{\mathrm{row},0}
 \le Ca^{-3}\norm{U-U_0}_{\cX}.
 \label{eq:Ndiff}
\end{equation}
Finally, for every $y\in\C^m$,
\begin{align}
 \norm{N(U)y}_2&\le Ca^{-1}\norm{y}_2,
 \label{eq:Nadm1}\\
 \norm{W_0^{-1/2}N(U)y}_\infty
 &\le Ca^{-1}\norm{W_0^{-1/2}y}_\infty+Ca^{-2}\norm{y}_2.
 \label{eq:Nadm2}
\end{align}
\end{lemma}

\begin{proof}
At $U_0$, $G_0$ is real symmetric and
$0\preceq G_0\preceq\Id$. Thus the smallest singular value of
$a\Id+c_pG_0$ is at least $a$. Lemma~\ref{lem:rowmap} and a Neumann
series prove \eqref{eq:Rop} when $C\delta\le a/2$. Here and below
$1/2\le c_p<1$, because $p\ge4$.

The identity
\begin{equation}
 R=a^{-1}\Id-\frac {c_p}a GR
 \label{eq:Ridentity}
\end{equation}
will be used for row estimates. Since $G=\Phi\Phi^\transpose$, Lemma~\ref{lem:rowmap} gives
\begin{equation}
 \norm{e_i^\transpose G}_2
 \le\norm{\phi_i}_2\norm{\Phi^\transpose}_{\op}
 \le C\sqrt{w_{0,i}}.
 \label{eq:Grow}
\end{equation}
Applying \eqref{eq:Ridentity} to $X$, and using
$\norm{RX}_{\op}\le2a^{-1}\norm{X}_{\op}$, yields row by row
\[
 \frac{\norm{e_i^\transpose RX}_2}{\sqrt{w_{0,i}}}
 \le a^{-1}\frac{\norm{e_i^\transpose X}_2}{\sqrt{w_{0,i}}}
 +\frac{c_p}{a}
 \frac{\norm{e_i^\transpose G}_2}{\sqrt{w_{0,i}}}
 \norm{RX}_{\op}
 \le a^{-1}\norm{X}_{\mathrm{row},0}
 +Ca^{-2}\norm{X}_{\op}.
\]

Because $a+c_p=1$ and $R$ commutes with $G$,
\begin{equation}
 N(U)=2(\Id-G)R=\frac2{c_p}(R-\Id).
 \label{eq:Nres}
\end{equation}
The resolvent identity gives
\begin{equation}
 N(U)-N(U_0)
 =-2R(U)(G(U)-G_0)R(U_0).
 \label{eq:Ndiffexact}
\end{equation}
The operator part of \eqref{eq:Ndiff} follows from \eqref{eq:GLip} and \eqref{eq:Rop}. For the row part, the right resolvent first gives
\[
 \norm{(G-G_0)R_0}_{\op}
 \le Ca^{-1}\norm{U-U_0}_{\cX}
\]
and, directly on each row,
\[
 \norm{(G-G_0)R_0}_{\mathrm{row},0}
 \le\norm{G-G_0}_{\mathrm{row},0}\norm{R_0}_{\op}
 \le Ca^{-1}\norm{U-U_0}_{\cX}.
\]
Apply \eqref{eq:Rrow} to
\[
 X=(G-G_0)R_0.
\]
The two terms are respectively bounded by
$Ca^{-2}\norm{U-U_0}_{\cX}$ and
$Ca^{-3}\norm{U-U_0}_{\cX}$. This proves the row part of
\eqref{eq:Ndiff}; the weaker common bound $Ca^{-3}$ also covers the
operator part since $a\le1/2$. Finally, for a one-column matrix $y$,
\[
 \norm{y}_{\mathrm{row},0}
 =\norm{W_0^{-1/2}y}_\infty,\qquad
 \norm{y}_{\op}=\norm{y}_2.
\]
Thus \eqref{eq:Nres}, \eqref{eq:Rop}, \eqref{eq:Rrow}, and
$2/c_p\le4$ give \eqref{eq:Nadm1}--\eqref{eq:Nadm2}.
\end{proof}

\begin{proposition}[Identification of the horizontal physical frame]
\label{prop:physicalFrame}
Let $y_\tau=x+\tau e$ range over a real interval
$I\ni0$ contained in $\operatorname{int}K$. The physical horizontal
frame constructed above is the unique real analytic orthonormal frame
for $\operatorname{col}(\widehat B_\tau)$ with initial value
\eqref{eq:physicalU0} and gauge \eqref{eq:horizontal}. Every row of
this frame is nonzero. On every
subinterval on which it lies in the row-map neighborhood of
Lemma~\ref{lem:rowmap}, it solves \eqref{eq:frameODE}.
Conversely, a real solution of \eqref{eq:frameODE} with initial value
\eqref{eq:physicalU0} agrees with the physical horizontal frame on
every connected interval on which it remains in that neighborhood.
If a finite endpoint lies in $I$ and the solution remains in a closed
subball of the neighborhood up to that endpoint, the identification
continues across it.
\end{proposition}

\begin{proof}
Lemma~\ref{lem:LewisAnalytic} shows that
$\widehat B_\tau=W_{y_\tau}^\beta A_{y_\tau}$ is real analytic and
full column rank on $I$. The analytic orthonormal frame
\[
 \widetilde U_\tau
 =\widehat B_\tau
 \bigl(\widehat B_\tau^\transpose\widehat B_\tau\bigr)^{-1/2}
\]
spans its column space and satisfies
$\widetilde U_0=\widehat B_0=U_0$. Applying the gauge construction in
Lemma~\ref{lem:horizontalgauge}, with initial gauge $\Id_d$, gives the
stated frame. If $U$ and $\overline U$ are two such frames, then
$\overline U=UO$ for an orthogonal matrix path $O$ with $O(0)=\Id$.
Their horizontal conditions give $O'=0$, so $O=\Id$; this proves
uniqueness without invoking any choice of orthonormalization.
Lemma~\ref{lem:physicalFrameAlgebra} gives the nonzero-row claim and
\eqref{eq:physicalGN} identifies its row maps with the physical Lewis
quantities.

The slack identity \eqref{eq:sexplicit} gives
$A_{y_\tau}'=-\Diag(s(\tau))A_{y_\tau}$. Moreover,
\eqref{eq:LewisDerivative}, followed by
\eqref{eq:physicalGN}, gives
\[
 \frac{\dd}{\dd\tau}\log w_{y_\tau}
 =-W(U_\tau)^{-1/2}N(U_\tau)v(U_\tau,\tau).
\]
The product rule therefore gives
\[
 \widehat B_\tau'
 =\Diag\!\left(
 -s(\tau)-\beta W(U_\tau)^{-1/2}
 N(U_\tau)v(U_\tau,\tau)\right)\widehat B_\tau.
\]
The horizontal-gauge formula in Lemma~\ref{lem:horizontalgauge} now
yields precisely \eqref{eq:frameODE}. Lemmas~\ref{lem:rowmap} and
\ref{lem:resolvent} make its right side real analytic, hence locally
Lipschitz in the $\cX$ norm, throughout the stated neighborhood.
Thus a real solution of the closed equation and the physical frame
coincide near zero by ODE uniqueness. Reapplying uniqueness at every
point of their common interval proves identification on its connected
component. If a finite endpoint in $I$ is approached inside a closed
subball, the rational slack functions are regular there and all
$U$-dependent denominators stay uniformly away from zero. The vector
field is bounded on a slightly larger compact neighborhood, so the
closed solution has a limit. Local existence and uniqueness at that
limiting point extend both solutions and the identification across
the endpoint.
\end{proof}

\begin{remark}\label{rem:Nrowfalse}
A uniform bound on $\norm{N(U)}_{\mathrm{row},0}$ is generally false. On the null space of $G$, $N$ has an identity component of size comparable to $p$. Lemma~\ref{lem:resolvent} controls instead differences of $N$, derivatives of $N$, and the action of $N$ on vectors whose Euclidean and relative-coordinate norms are both controlled.
\end{remark}

\subsection{Complex existence and Cauchy estimates}

The elementary analytic ODE lemma used next is proved in Appendix~\ref{app:analyticODE}.

\begin{lemma}[Uniform analytic flow]\label{lem:analyticflow}
There are universal constants $c,C>0$ such that the solution of \eqref{eq:frameODE} extends holomorphically to
\begin{equation}
 \abs\tau<2\rho_p,
 \qquad \rho_p=cp^{-3}=c'a^3,
 \label{eq:rhop}
\end{equation}
and satisfies
\begin{equation}
 \norm{U(\tau)-U_0}_{\cX}\le c_1a/2.
 \label{eq:Uball}
\end{equation}
On the same disk,
\begin{equation}
 U(\tau)^\transpose U(\tau)=\Id_d.
 \label{eq:complexStiefel}
\end{equation}
\end{lemma}

\begin{proof}
Decrease $c_1$ if necessary so that $c_1a\le c_0$, where $c_0$ is
the row-map radius in Lemma~\ref{lem:rowmap}, and set
\[
 \delta=c_1a,\qquad
 \mathcal D=\{(U,\tau):
 \norm{U-U_0}_{\cX}<\delta,\ \abs\tau<1/4\}.
\]
For $(U,\tau)\in\mathcal D$, Lemma~\ref{lem:rowmap} gives
\begin{equation}
 2^{-1/2}\sqrt{w_{0,i}}
 \le\abs{\omega_i(U)}
 \le(3/2)^{1/2}\sqrt{w_{0,i}},
 \qquad 1\le i\le m.
 \label{eq:omegaCoordinateComparison}
\end{equation}
Indeed, $\omega_i(U)^2=u_i^\transpose u_i$ lies in the disk of
radius $w_{0,i}/2$ around $w_{0,i}$.

By \eqref{eq:sinitial}, $\abs{s_i(0)}\le1$. Hence
$\abs{1+\tau s_i(0)}\ge3/4$ on $\abs\tau\le1/4$, and
\eqref{eq:sexplicit} gives
\begin{equation}
 \norm{s(\tau)}_\infty\le2,
 \qquad \norm{W_0^{1/2}s(\tau)}_2\le2.
 \label{eq:scomplex}
\end{equation}
The second bound follows term by term from
\[
 \sum_iw_{0,i}\abs{s_i(\tau)}^2
 \le(4/3)^2\sum_iw_{0,i}\abs{s_i(0)}^2
 \le(4/3)^2.
\]
Since $v_i=\omega_i(U)s_i(\tau)$,
\eqref{eq:omegaCoordinateComparison} and \eqref{eq:scomplex} imply
\begin{equation}
 \norm{v}_2
 \le C\norm{W_0^{1/2}s(\tau)}_2\le C,
 \qquad
 \norm{W_0^{-1/2}v}_\infty
 \le C\norm{s(\tau)}_\infty\le C.
 \label{eq:vcomplex}
\end{equation}

Put $y=N(U)v$. Equations \eqref{eq:Nadm1},
\eqref{eq:Nadm2} and \eqref{eq:vcomplex} give
\begin{equation}
 \norm{y}_2\le Ca^{-1},
 \qquad
 \norm{W_0^{-1/2}y}_\infty\le Ca^{-2}.
 \label{eq:NvComplex}
\end{equation}
The diagonal operator
$W_0^{1/2}W(U)^{-1/2}$ has operator norm at most $C$ by
\eqref{eq:omegaCoordinateComparison}. Consequently
\[
 \norm{W(U)^{-1/2}y}_\infty
 \le C\norm{W_0^{-1/2}y}_\infty\le Ca^{-2},
 \qquad
 \norm{W_0^{1/2}W(U)^{-1/2}y}_2
 \le C\norm{y}_2\le Ca^{-1}.
\]
Since $0<\beta<1/2$, \eqref{eq:zclosed} and
\eqref{eq:scomplex} now give separately in the two relevant norms
\begin{equation}
 \norm{z}_\infty\le Ca^{-2},
 \qquad \norm{W_0^{1/2}z}_2\le Ca^{-1}.
 \label{eq:zcomplexbounds}
\end{equation}

Let $Z=\Diag(z)$. Since $U$ remains in a fixed operator-norm neighborhood of $U_0$,
\[
 \norm{u_i}_2\le\norm{u_{0,i}}_2+\norm{u_i-u_{0,i}}_2
 \le(1+\delta)\sqrt{w_{0,i}}\le C\sqrt{w_{0,i}},
\]
and therefore
\[
 \norm{ZU}_{\mathrm F}^2
 =\sum_i\abs{z_i}^2\norm{u_i}_2^2
 \le C\norm{W_0^{1/2}z}_2^2\le Ca^{-2}.
\]
For the relative row norm,
\begin{equation}
e_i^\transpose\cF(U,\tau)
 =z_i u_i^\transpose-u_i^\transpose(U^\transpose ZU).
 \label{eq:Frowidentity}
\end{equation}
Here
\[
 \norm{U}_{\op}
 \le\norm{U_0}_{\op}+\norm{U-U_0}_{\mathrm F}
 \le1+c_1a\le C
\]
and, by \eqref{eq:complexTransposeNorm},
\[
 \norm{U^\transpose ZU}_{\op}
 \le\norm{U^\transpose}_{\op}\norm{ZU}_{\op}
 \le C\norm{ZU}_{\mathrm F}\le Ca^{-1}.
\]
Dividing \eqref{eq:Frowidentity} by $\sqrt{w_{0,i}}$ and using the
last two displays gives, for every $i$,
\[
 \frac{\norm{e_i^\transpose\cF(U,\tau)}_2}
 {\sqrt{w_{0,i}}}
 \le C\abs{z_i}+C\norm{U^\transpose ZU}_{\op}
 \le Ca^{-2}.
\]
The Frobenius estimate follows directly from
\[
 \norm{\cF(U,\tau)}_{\mathrm F}
 \le\bigl(1+\norm{U}_{\op}^2\bigr)\norm{ZU}_{\mathrm F}
 \le Ca^{-1}.
\]
Thus
\begin{equation}
 \sup_{\norm{U-U_0}_{\cX}<\delta,\ \abs\tau<1/4}
 \norm{\cF(U,\tau)}_{\cX}\le Ca^{-2}.
 \label{eq:Fbound}
\end{equation}
On $\mathcal D$, every denominator
$1+\tau s_i(0)$ is nonzero, every chosen square-root branch remains
inside its defining disk, and
$a\Id+c_pG(U)$ is invertible by Lemma~\ref{lem:resolvent}.
Therefore $s$, $\omega$, $W(U)^{-1/2}$, $N$, $v$, $z$ and
$\cF$ are jointly holomorphic on $\mathcal D$. Applying
Lemma~\ref{lem:abstractODE} with
\[
 \tau_0=\frac14,\qquad \delta=c_1a,\qquad M=Ca^{-2}
\]
gives
\[
 \min\{\tau_0,\delta/M\}
 \ge c_2a^3.
\]
Here $a\le1/2$, and every constant used above is universal:
$\norm{U_0}_{\op}=1$, while the relative-row normalization absorbs
all dependence on $m,d$, and $\min_iw_{0,i}$. Thus the ODE lemma gives
a holomorphic solution on a disk of radius $c_3a^3$ that stays in the
half ball. Define $\rho_p=c_3a^3/2$, after decreasing $c_3$ so that
$2\rho_p<1/4$. This proves \eqref{eq:Uball} on
$\abs\tau<2\rho_p$, with a dimension-free constant.

To verify the complex Stiefel constraint, set $Q=U^\transpose U$ and
$\mathcal M=U^\transpose ZU$. Directly from \eqref{eq:frameODE},
\[
 Q'=\mathcal M(\Id-Q)+(\Id-Q)\mathcal M.
\]
Thus $E=\Id-Q$ solves the homogeneous linear matrix equation
$E'=-\mathcal ME-E\mathcal M$ with $E(0)=0$. Uniqueness gives
$E\equiv0$, proving \eqref{eq:complexStiefel} throughout the disk.
\end{proof}

\begin{lemma}[Kernel identity on the complex flow]\label{lem:kernel}
On the disk of Lemma~\ref{lem:analyticflow},
\begin{equation}
 G(U(\tau))\omega(U(\tau))=\omega(U(\tau)),
 \qquad N(U(\tau))\omega(U(\tau))=0.
 \label{eq:Nomega}
\end{equation}
\end{lemma}

\begin{proof}
By \eqref{eq:complexStiefel}, $\sum_j u_ju_j^\transpose=U^\transpose U=\Id$. Therefore
\[
 (G\omega)_i
 =\frac1{\omega_i}\sum_j(u_i^\transpose u_j)^2
 =\frac1{\omega_i}u_i^\transpose
 \left(\sum_j u_ju_j^\transpose\right)u_i
 =\omega_i.
\]
Thus $G\omega=\omega$. Since $a+c_p=1$, we also have
$R\omega=(a\Id+c_pG)^{-1}\omega=\omega$, and
\eqref{eq:Nframe} gives $N\omega=2(\Id-G)R\omega=0$.
\end{proof}

\begin{theorem}[Analytic horizontal-frame jets]\label{thm:jets}
Write $\omega_x=w_x^{1/2}$ coordinatewise. There are universal
constants $c,C>0$ such that, for every real base point
$x\in\operatorname{int}K$, every $k\ge1$, and every
$r_1,\ldots,r_k\in\R^d$,
\begin{align}
 \norm{D^kN_x[r_1,\ldots,r_k]}_{\op}
 &+\norm{W_x^{-1/2}D^kN_x[r_1,\ldots,r_k]}_{2\to\infty}
 \notag\\
 &\le Cp^2 k^k\rho_p^{-k}
 \prod_{j=1}^k\norm{r_j}_{g_0(x)},
 \label{eq:mixedNjet}\\
 \norm{D^k\omega_x[r_1,\ldots,r_k]}_2
 &\le C k^k\rho_p^{-k}
 \prod_{j=1}^k\norm{r_j}_{g_0(x)}.
 \label{eq:mixedomegajet}
\end{align}
For repeated directions,
\begin{align}
 \norm{D^kN_x[r^k]}_{\op}
 +\norm{W_x^{-1/2}D^kN_x[r^k]}_{2\to\infty}
 &\le Cp^2 k!\rho_p^{-k}\norm{r}_{g_0(x)}^k,
 \label{eq:diagNjet}\\
 \norm{D^k\omega_x[r^k]}_2
 &\le C k!\rho_p^{-k}\norm{r}_{g_0(x)}^k.
 \label{eq:diagomegajet}
\end{align}
\end{theorem}

\begin{proof}
We first prove the repeated-direction estimates. Given $x$ and
$r\ne0$, use the affine map
\[
 \Psi(\widetilde x)=x+g_0(x)^{-1/2}\widetilde x.
\]
Then $\widetilde g_0(0)=\Id_d$ and the direction corresponding to
$r$ is $\widetilde r=g_0(x)^{1/2}r$, whose Euclidean norm is
$\norm{r}_{g_0(x)}$. Equations \eqref{eq:NjetCovariance} and
\eqref{eq:localNormCovariance}, together with invariance of the Lewis
weights at corresponding points, show that it suffices by homogeneity
to work under \eqref{eq:normalization} with $\norm{r}=1$.

Let $U(\tau)$ be the analytic solution in
Lemma~\ref{lem:analyticflow}, with initial frame
\eqref{eq:physicalU0}. For real $\tau$ near zero,
Proposition~\ref{prop:physicalFrame} and real ODE uniqueness give
\begin{equation}
 N(U(\tau))=N_{x+\tau r},
 \qquad
 \omega(U(\tau))=w_{x+\tau r}^{1/2}.
 \label{eq:jetIdentification}
\end{equation}
Indeed, \eqref{eq:sinitial} implies
$1+\tau s_i(0)>0$ for real $\abs\tau<2\rho_p<1/4$, so this real
segment remains in $\operatorname{int}K$ and the physical frame is
defined there. The local identification near zero gives
\[
 \left.\frac{\dd^k}{\dd\tau^k}N(U(\tau))\right|_{\tau=0}
 =D^kN_x[r,\ldots,r],
 \qquad
 \left.\frac{\dd^k}{\dd\tau^k}\omega(U(\tau))\right|_{\tau=0}
 =D^k\omega_x[r,\ldots,r].
\]

Since \eqref{eq:Uball} places the circle
$\abs\tau=\rho_p$ in the resolvent neighborhood,
Lemma~\ref{lem:resolvent} gives
\[
 \sup_{\abs\tau\le\rho_p}
 \left(\norm{N(U(\tau))-N(U_0)}_{\op}
 +\norm{N(U(\tau))-N(U_0)}_{\mathrm{row},0}\right)
 \le Ca^{-3}(c_1a/2)\le Cp^2.
\]
For $k\ge1$, subtracting $N(U_0)$ does not change the derivative.
Cauchy's formula on the circle $\abs\tau=\rho_p$, strictly inside
$\abs\tau<2\rho_p$, in the matrix Banach norm
$\norm{M}_{\op}+\norm{W_x^{-1/2}M}_{2\to\infty}$ gives the following
estimate; here its second summand is exactly
$\norm{M}_{\mathrm{row},0}$:
\[
 \norm{\frac{\dd^k}{\dd\tau^k}N(U(\tau))\big|_{\tau=0}}_{\op}
 +\norm{W_x^{-1/2}
 \frac{\dd^k}{\dd\tau^k}N(U(\tau))\big|_{\tau=0}}_{2\to\infty}
 \le Cp^2k!\rho_p^{-k},
\]
which proves \eqref{eq:diagNjet}. Equation \eqref{eq:omegaLip} and
\eqref{eq:Uball} give
$\sup_{\abs\tau\le\rho_p}
\norm{\omega(U(\tau))-\omega(U_0)}_2\le Ca\le C$.
The same Cauchy argument proves \eqref{eq:diagomegajet} without a
prefactor depending on $p$.

It remains to pass from diagonal to mixed directions. The maps
$D^kN_x$ and $D^k\omega_x$ are symmetric $k$-linear maps into their
respective Banach spaces. For unit vectors $r_1,\ldots,r_k$, the real
polarization identity and the triangle inequality give
\[
 \norm{T[r_1,\ldots,r_k]}
 \le\frac1{2^kk!}\sum_{\epsilon\in\{-1,1\}^k}
 \norm{T[(\epsilon_1r_1+\cdots+\epsilon_kr_k)^k]}
 \le\frac{k^k}{k!}
 \sup_{\norm{r}\le1}\norm{T[r^k]}.
\]
Scaling each input separately gives the same bound multiplied by
$\prod_j\norm{r_j}$. Combining this with the diagonal $k!$ bounds
and undoing the affine normalization proves
\eqref{eq:mixedNjet} and \eqref{eq:mixedomegajet}.
\end{proof}

\section{Projection-weighted Wick contractions}\label{sec:wick}

Kook \cite[Section~3.3]{Kook26} estimated the first Wiener-chaos
contractions case by case. We isolate their Parseval-frame structure
in an all-order theorem whose double-contraction bound removes the
residual $d^{1/2}$ loss. Throughout this section all tensors are real,
and $\Tr$ denotes the unnormalized trace
\[
 (\Tr^rT)_{a_1\ldots a_{n-2r}}
 =\sum_{b_1,\ldots,b_r=1}^d
 T_{b_1b_1\ldots b_rb_ra_1\ldots a_{n-2r}}.
\]
For tensors $S,T$, the contraction $S\otimes_\ell T$ sums the last
$\ell$ indices of each tensor, and
$S\widetilde\otimes_\ell T$ is its average over all permutations of
the remaining indices. Thus symmetrization is an orthogonal
projection in Hilbert--Schmidt norm and introduces no additional
combinatorial factor. Let $I_n(T)$ denote the order-$n$ Wick
polynomial, normalized by
\begin{equation}
 \E[I_n(S)I_m(T)]=\mathbf 1_{\{m=n\}}n!\ip{S}{T}_{\HS}.
 \label{eq:WickIsometry}
\end{equation}
The raw-to-Wick and product formulas
\cite[Chapter~1]{Nualart06} are
\begin{align}
 T[h^{\otimes n}]
 &=\sum_{r=0}^{\lfloor n/2\rfloor}
 a_{n,r}I_{n-2r}(\Tr^rT),
 \qquad a_{n,r}=\frac{n!}{2^rr!(n-2r)!},
 \label{eq:rawWick}\\
 I_a(S)I_b(T)
 &=\sum_{\ell=0}^{a\wedge b}
 \ell!\binom a\ell\binom b\ell
 I_{a+b-2\ell}(S\widetilde\otimes_\ell T),
 \label{eq:WickProduct}
\end{align}
where $\widetilde\otimes_\ell$ is the symmetrized $\ell$-fold contraction. Appendix~\ref{app:wick} records these identities.
We use the conventions that $I_0(c)=c$, a contraction whose order
exceeds one of the tensor orders is absent, and the Hilbert--Schmidt
norm of a scalar is its absolute value.  These conventions include
without exception the cases in which a trace leaves a tensor of
order zero or one.

Let $U\in\R^{m\times d}$ satisfy $U^\transpose U=\Id_d$ and assume
that every row is nonzero. Write the column representation of its
$i$th row as
\begin{equation}
 u_i=\sqrt{w_i}\,k_i,\qquad w_i=\norm{u_i}^2>0,
 \qquad \norm{k_i}=1,
 \qquad \sum_iw_i=d,
 \label{eq:parsevalrows}
\end{equation}
and set $K_i=k_i\otimes k_i$. Define
\begin{equation}
 \Gamma_{ij}=\sqrt{w_iw_j}\ip{K_i}{K_j}_{\HS}
 =\frac{P_{ij}^2}{\sqrt{w_iw_j}},
 \qquad P=UU^\transpose.
 \label{eq:Gamma}
\end{equation}
The Schur-product theorem \cite[Section~5.2]{HornJohnson13} gives
\begin{equation}
 0\preceq\Gamma\preceq\Id_m.
 \label{eq:GammaContraction}
\end{equation}
Indeed, $P\preceq\Id$ implies $P\had(\Id-P)\succeq0$, which is $W-P^{\had2}\succeq0$; conjugation by $W^{-1/2}$ proves \eqref{eq:GammaContraction}.

\begin{theorem}[Weighted Wick contraction]\label{thm:weightedWick}
Let $n\ge1$ and let $T_i\in\operatorname{Sym}^n(\R^d)$, $1\le i\le m$. Suppose that, for some $L\ge0$,
\begin{align}
 \left(\sum_i\abs{T_i[x^{\otimes n}]}^2\right)^{1/2}
 &\le L\norm{x}^n,
 \label{eq:WW1}\\
 \abs{T_i[x_1,\ldots,x_n]}
 &\le L\sqrt{w_i}\prod_{j=1}^n\norm{x_j},
 \label{eq:WW2}\\
 \abs{\sum_i\sqrt{w_i}T_i[x^{\otimes n}]}
 &\le L\norm{x}^n
 \label{eq:WW3}
\end{align}
for all indicated vectors. If $h\sim N(0,\Id_d)$, then
\begin{equation}
 \left\Vert
 \sum_i\sqrt{w_i}(k_i^\transpose h)^2T_i[h^{\otimes n}]
 \right\Vert_{L^2}
 \le (Cn)^{Cn}L(d+n)^{n/2}.
 \label{eq:WWconclusion}
\end{equation}
\end{theorem}

\begin{proof}
Set
\[
 \cP(h)=\sum_i\sqrt{w_i}(k_i^\transpose h)^2T_i[h^{\otimes n}].
\]
Since $\norm{k_i}=1$,
\begin{equation}
 (k_i^\transpose h)^2=1+I_2(K_i).
 \label{eq:quadWick}
\end{equation}
The noncentered part
\[
 \cP_0(h)=\sum_i\sqrt{w_i}T_i[h^{\otimes n}]
\]
satisfies, by \eqref{eq:WW3} and the Gaussian radial moment bound,
\begin{equation}
 \norm{\cP_0}_{L^2}
 \le L(\E\norm{h}^{2n})^{1/2}
 \le (Cn)^{n/2}L(d+n)^{n/2}.
 \label{eq:P0bound}
\end{equation}

For the centered part, put $S_{i,r}=\Tr^rT_i$ and $m_r=n-2r$. Then
\begin{equation}
 \cP_c
 =\sum_{r=0}^{\lfloor n/2\rfloor}a_{n,r}
 \sum_i\sqrt{w_i}I_2(K_i)I_{m_r}(S_{i,r}).
 \label{eq:Pcraw}
\end{equation}
Only the contractions $\ell=0,1,2$ occur. The $\ell=1$ term is absent
when $m_r=0$, and the $\ell=2$ term is absent when $m_r<2$.
Thus, when $n=1$, only $\ell=0,1$ occur, while for $n=2$ the
$\ell=2$ contraction is a scalar.  The estimates below use the same
formulas in these edge cases.

For $\ell=0$, an unsymmetrized kernel is
\[
 A_{r,0}=\sum_i\sqrt{w_i}K_i\otimes S_{i,r}.
\]
Using \eqref{eq:GammaContraction} coordinate by coordinate,
\begin{equation}
 \norm{A_{r,0}}_{\HS}^2
 =\sum_{i,j}\Gamma_{ij}\ip{S_{i,r}}{S_{j,r}}_{\HS}
 \le\sum_i\norm{S_{i,r}}_{\HS}^2.
 \label{eq:l0}
\end{equation}
Symmetrization cannot increase the Hilbert--Schmidt norm.

For $\ell=1$, an unsymmetrized representative is
\[
 A_{r,1}=\sum_i u_i\otimes S_{i,r}(k_i,\cdot).
\]
For a multi-index
$\gamma=(a_1,\ldots,a_{m_r-1})$, put
\[
 c_{i,\gamma}
 =\sum_{b=1}^d(k_i)_b(S_{i,r})_{b,a_1,\ldots,a_{m_r-1}}.
\]
Then $(A_{r,1})_{a_0,\gamma}=\sum_i(u_i)_{a_0}c_{i,\gamma}$.
For each fixed $\gamma$, the Parseval identity for the rows of $U$
gives
\[
 \sum_{a_0=1}^d
 \left|\sum_i(u_i)_{a_0}c_{i,\gamma}\right|^2
 =c_\gamma^\transpose Pc_\gamma
 \le\sum_i|c_{i,\gamma}|^2.
\]
Summing over $\gamma$ and using $\norm{k_i}=1$ gives
\begin{equation}
 \norm{A_{r,1}}_{\HS}^2
 \le\sum_i\norm{S_{i,r}(k_i,\cdot)}_{\HS}^2
 \le\sum_i\norm{S_{i,r}}_{\HS}^2.
 \label{eq:l1}
\end{equation}
The same bound holds after symmetrization.

For $\ell=2$, which is present only when $m_r\ge2$, the kernel is
\[
 A_{r,2}=\sum_i\sqrt{w_i}S_{i,r}(k_i,k_i,\cdot).
\]
For a fixed output multi-index $(a_1,\ldots,a_{m_r-2})$, expanding the $r$ traces and applying \eqref{eq:WW2} gives
\begin{align*}
 &\abs{S_{i,r}(k_i,k_i,e_{a_1},\ldots,e_{a_{m_r-2}})}\\
 &\quad\le
 \sum_{b_1,\ldots,b_r=1}^d
 \abs{T_i[e_{b_1},e_{b_1},\ldots,e_{b_r},e_{b_r},
 k_i,k_i,e_{a_1},\ldots]}
 \le Ld^r\sqrt{w_i}.
\end{align*}
Thus each output coordinate of $A_{r,2}$ is at most
\begin{equation}
 Ld^r\sum_iw_i=Ld^{r+1}.
 \label{eq:l2coord}
\end{equation}
There are $d^{m_r-2}$ output coordinates, and hence
\begin{equation}
 \norm{A_{r,2}}_{\HS}
 \le Ld^{r+1}d^{(m_r-2)/2}=Ld^{n/2},
 \label{eq:l2}
\end{equation}
because
\begin{equation}
 r+1+\frac{n-2r-2}{2}=\frac n2.
 \label{eq:dimensionidentity}
\end{equation}

It remains to track the coefficients. Regard
$\mathbf T(h)=(T_i[h^{\otimes n}])_{i=1}^m$ as an
$\R^m$-valued polynomial. The coordinatewise raw-to-Wick identity is
\[
 \mathbf T(h)
 =\sum_{r=0}^{\lfloor n/2\rfloor}
 a_{n,r}\bigl(I_{m_r}(S_{i,r})\bigr)_{i=1}^m.
\]
The chaos orders $m_r$ are pairwise distinct. Orthogonality and
\eqref{eq:WickIsometry} therefore give the exact identity
\begin{equation}
 \E\norm{\mathbf T(h)}_2^2
 =\sum_{r=0}^{\lfloor n/2\rfloor}
 a_{n,r}^2m_r!\sum_i\norm{S_{i,r}}_{\HS}^2.
 \label{eq:tensorenergyExact}
\end{equation}
On the other hand, \eqref{eq:WW1} and the Gaussian radial-moment
bound imply
\begin{equation}
 \sum_{r=0}^{\lfloor n/2\rfloor}
 a_{n,r}^2m_r!\sum_i\norm{S_{i,r}}_{\HS}^2
 \le L^2(C(d+n))^n.
 \label{eq:tensorenergy}
\end{equation}
This is the only point in the proof at which orthogonality between
different values of $r$ is used.  After multiplication by
$I_2(K_i)$, the output order is
\begin{equation}
 q(r,\ell)=m_r+2-2\ell=n-2r+2-2\ell,
 \label{eq:outputChaosOrder}
\end{equation}
and different pairs $(r,\ell)$ can have the same value of
$q(r,\ell)$.  We therefore make no orthogonality assertion for the
terms in \eqref{eq:Pcraw}; from this point onward they are combined
by the triangle inequality.

For $\ell=0$, the $L^2$ norm of the contribution for a fixed $r$ is at most
\[
 a_{n,r}\sqrt{(m_r+2)!}
 \left(\sum_i\norm{S_{i,r}}_{\HS}^2\right)^{1/2}
 \le (n+2)
 \left(a_{n,r}^2m_r!\sum_i\norm{S_{i,r}}_{\HS}^2\right)^{1/2}.
\]
For $\ell=1$, the product coefficient is $2m_r$, the output order is $m_r$, and the corresponding bound is at most
\[
 2m_r
 \left(a_{n,r}^2m_r!\sum_i\norm{S_{i,r}}_{\HS}^2\right)^{1/2}.
\]
Taking the triangle inequality first, and then Cauchy--Schwarz in
$r$, the combined $\ell=0,1$ contribution is at most
\[
 3(n+2)\sqrt{1+\lfloor n/2\rfloor}\,
 \left(
 \sum_{r=0}^{\lfloor n/2\rfloor}
 a_{n,r}^2m_r!\sum_i\norm{S_{i,r}}_{\HS}^2
 \right)^{1/2}.
\]
Here the absent $\ell=1$ term for $m_r=0$ is simply omitted.  By
\eqref{eq:tensorenergy}, the last display is bounded by
$(Cn)^{Cn}L(d+n)^{n/2}$.

For $\ell=2$, the product coefficient is $m_r(m_r-1)$ and the output order is $m_r-2$. Equations \eqref{eq:l2} and the elementary estimates
\[
 a_{n,r}\le(Cn)^n,
 \qquad m_r(m_r-1)\le n^2,
 \qquad \sqrt{(m_r-2)!}\le(Cn)^{n/2}
\]
show that each such contribution is at most $(Cn)^{Cn}Ld^{n/2}$. There are at most $1+\lfloor n/2\rfloor$ values of $r$. The triangle inequality now proves the same bound for $\cP_c$. Combining it with \eqref{eq:P0bound} proves \eqref{eq:WWconclusion}.
\end{proof}

\begin{remark}
Condition \eqref{eq:WW1} controls the global tensor energy,
\eqref{eq:WW2} controls each row relative to its leverage weight, and
\eqref{eq:WW3} supplies the weighted mean cancellation. Thus the
theorem is tailored to Gaussian polynomials indexed by a Parseval
frame, rather than to arbitrary weighted families.

The double contraction is the only point at which the number of rows could enter. The factor $\sqrt{w_i}$ in \eqref{eq:WW2} turns the row sum into $\sum_iw_i=d$, and \eqref{eq:dimensionidentity} exactly compensates for the two contracted Gaussian indices.
\end{remark}

\section{Base-point bottleneck polynomials}\label{sec:base}

This section combines Theorems~\ref{thm:jets} and
\ref{thm:weightedWick} to control the base-point terms $H_j(0)$ in
the selective expansion \eqref{eq:selective-exact}.

Return to the normalization \eqref{eq:normalization}. Write
\begin{equation}
 U=W^\beta A_x,
 \qquad U^\transpose U=\Id_d,
 \qquad B=W^{1/2}A_x=W^{1/p}U,
 \qquad \norm{B}_{\op}\le1.
 \label{eq:baseUB}
\end{equation}
Let $u_i=\sqrt{w_i}k_i$ be the rows of $U$. For $h\sim N(0,\Id_d)$,
\begin{equation}
 q_i=\sqrt{w_i}(k_i^\transpose h)^2,
 \qquad v=Bh.
 \label{eq:baseqv}
\end{equation}

\begin{lemma}[Cubic base term]\label{lem:cubic}
There is a universal $C>0$ such that
\begin{equation}
 \norm{E_1(0)+\beta H_1(0)}_{L^2}
 \le Cp^2\sqrt d.
 \label{eq:cubicL2}
\end{equation}
\end{lemma}

\begin{proof}
Let $u=Uh$ and define
\begin{equation}
 M=-W^{-1/2}(\Id+\beta N)W^{1/p}U.
 \label{eq:Mdef}
\end{equation}
Since $W^{-\beta}=W^{-1/2}W^{1/p}$, substitution into \eqref{eq:Fprime} gives
\begin{equation}
 E_1(0)+\beta H_1(0)=-\ip{Mh}{u^{\had2}}=:-T(h).
 \label{eq:Tcubic}
\end{equation}
The polynomial $T$ is odd and centered. The Gaussian Poincar\'e
inequality \cite{BrascampLieb76,Ledoux01} gives
\begin{equation}
 \E T^2\le\E\norm{\nabla T(h)}^2
 \le2\E\norm{M^\transpose u^{\had2}}^2
 +8\E\norm{U^\transpose(u\had Mh)}^2.
 \label{eq:PoincareT}
\end{equation}
If $m_i^\transpose$ is the $i$th row of $M$, then
\begin{equation}
 \E\norm{U^\transpose(u\had Mh)}^2
 \le\sum_i\E[u_i^2(m_i^\transpose h)^2]
 \le3\sum_iw_i\norm{m_i}^2
 =3\tr(M^\transpose WM).
 \label{eq:secondcub}
\end{equation}
Let $C_M=MM^\transpose$. Since $u$ is Gaussian with covariance $P=UU^\transpose$,
\begin{align}
 \E\norm{M^\transpose u^{\had2}}^2
 &=w^\transpose C_Mw+2\tr(C_MP^{\had2})\notag\\
 &\le\norm{M^\transpose w}^2+2\tr(M^\transpose WM),
 \label{eq:firstcub}
\end{align}
where $P^{\had2}\preceq W$ was used.

Now
\[
 W^{1/2}M=-(\Id+\beta N)W^{1/p}U.
\]
Using $0\preceq N\preceq p\Id$, $W^{1/p}\preceq\Id$ and $\norm{U}_{\mathrm F}^2=d$,
\begin{equation}
 \tr(M^\transpose WM)\le Cp^2d.
 \label{eq:Mtrace}
\end{equation}
Also, using $P\preceq\Id$,
\begin{align}
 \norm{M^\transpose w}^2
 &=w^\transpose W^{-1/2}(\Id+\beta N)W^{1/p}
 P W^{1/p}(\Id+\beta N)W^{-1/2}w\notag\\
 &\le Cp^2w^\transpose W^{-1}w
 =Cp^2\sum_iw_i=Cp^2d.
 \label{eq:Mw}
\end{align}
Combining \eqref{eq:PoincareT}--\eqref{eq:Mw} proves \eqref{eq:cubicL2}.
\end{proof}

For $j\ge2$, define the symmetric $j$-tensor
\begin{equation}
 \cT_i^{(j)}[x_1,\ldots,x_j]
 =\frac1j\sum_{r=1}^j
 e_i^\transpose D^{j-1}N_x
 [x_1,\ldots,\widehat{x_r},\ldots,x_j]Bx_r.
 \label{eq:Tij}
\end{equation}
On the diagonal,
\begin{equation}
 \cT_i^{(j)}[h^{\otimes j}]
 =e_i^\transpose D^{j-1}N_x[h^{j-1}]Bh.
 \label{eq:Tdiag}
\end{equation}

\begin{theorem}[Natural scale of every bottleneck polynomial]\label{thm:Hbase}
For every $j\ge2$,
\begin{equation}
 \norm{H_j(0)}_{L^2}
 \le p^2(Cj)^{Cj}\rho_p^{-(j-1)}(d+j)^{j/2}.
 \label{eq:Hbase}
\end{equation}
In particular,
\begin{equation}
 \norm{H_4(0)}_{L^2}\le p^{O(1)}d^2.
 \label{eq:H4base}
\end{equation}
\end{theorem}

\begin{proof}
We verify the hypotheses of Theorem~\ref{thm:weightedWick} with $n=j$. First, Theorem~\ref{thm:jets}, \eqref{eq:Tdiag} and $\norm{B}_{\op}\le1$ give
\begin{equation}
 \left(\sum_i\abs{\cT_i^{(j)}[h^{\otimes j}]}^2\right)^{1/2}
 \le Cp^2(j-1)!\rho_p^{-(j-1)}\norm{h}^j.
 \label{eq:TWW1}
\end{equation}
Second, the mixed row estimate \eqref{eq:mixedNjet} gives
\begin{equation}
 \abs{\cT_i^{(j)}[x_1,\ldots,x_j]}
 \le Cp^2j^j\rho_p^{-(j-1)}\sqrt{w_i}
 \prod_{\ell=1}^j\norm{x_\ell}.
 \label{eq:TWW2}
\end{equation}

For the weighted mean, let $\omega=w^{1/2}$. The identity $N\omega=0$ holds along the complex line by Lemma~\ref{lem:kernel}. Differentiating it $j-1$ times in direction $h$ gives
\begin{equation}
 D^{j-1}N[h^{j-1}]\omega
 =-\sum_{\ell=0}^{j-2}\binom{j-1}{\ell}
 D^\ell N[h^\ell]
 D^{j-1-\ell}\omega[h^{j-1-\ell}].
 \label{eq:kernelderivative}
\end{equation}
To make the cancellation explicit, put $k=j-1$.  Entrywise
differentiation of the identity $N_y^\transpose=N_y$ shows that
every matrix $D^\ell N[h^\ell]$ is symmetric at a real base point.
In particular,
\begin{equation}
 \omega^\transpose D^kN[h^k]Bh
 =\bigl(D^kN[h^k]\omega\bigr)^\transpose Bh.
 \label{eq:weightedMeanSymmetry}
\end{equation}
Substituting \eqref{eq:kernelderivative} in
\eqref{eq:weightedMeanSymmetry}, and using
$\norm{B}_{\op}\le1$, gives
\begin{align}
 &\abs{\sum_i\sqrt{w_i}\cT_i^{(j)}[h^{\otimes j}]}
 =\abs{\omega^\transpose D^{j-1}N[h^{j-1}]Bh}\notag\\
 &\quad\le\norm{h}
 \sum_{\ell=0}^{j-2}\binom{j-1}{\ell}
 \norm{D^\ell N[h^\ell]}_{\op}
 \norm{D^{j-1-\ell}\omega[h^{j-1-\ell}]}_2.
 \label{eq:weightedMeanLeibniz}
\end{align}
The factorial cancellation in this Leibniz sum is exact.  Namely,
for $1\le\ell\le k-1$, \eqref{eq:diagNjet} and
\eqref{eq:diagomegajet} give
\begin{align}
 &\binom{k}{\ell}
 \norm{D^\ell N[h^\ell]}_{\op}
 \norm{D^{k-\ell}\omega[h^{k-\ell}]}_2\notag\\
 &\qquad\le
 Cp^2\binom{k}{\ell}\ell!(k-\ell)!
 \rho_p^{-k}\norm{h}^k
 =Cp^2k!\rho_p^{-k}\norm{h}^k.
 \label{eq:weightedMeanFactorials}
\end{align}
For $\ell=0$, the same conclusion, with $p$ in place of $p^2$,
follows from $\norm N_{\op}\le p$ and
\eqref{eq:diagomegajet}.  The case $k=1$, equivalently $j=2$,
consists only of this $\ell=0$ term.  Multiplication by the remaining
factor $\norm h$ in \eqref{eq:weightedMeanLeibniz}, followed by
summing at most $k$ terms, therefore yields
\begin{equation}
 \abs{\sum_i\sqrt{w_i}\cT_i^{(j)}[h^{\otimes j}]}
 \le Cp^2j!\rho_p^{-(j-1)}\norm{h}^j
 \le p^2(Cj)^{Cj}\rho_p^{-(j-1)}\norm{h}^j.
 \label{eq:TWW3}
\end{equation}
Thus \eqref{eq:WW1}--\eqref{eq:WW3} hold with
$L=p^2(Cj)^{Cj}\rho_p^{-(j-1)}$. Finally,
\[
 H_j(0)=\sum_i\sqrt{w_i}(k_i^\transpose h)^2
 \cT_i^{(j)}[h^{\otimes j}].
\]
Theorem~\ref{thm:weightedWick} proves \eqref{eq:Hbase}. Equation \eqref{eq:H4base} follows because $\rho_p^{-3}=p^{O(1)}$.
\end{proof}

\section{Pathwise localization and logarithmic summation}\label{sec:path}

Building on the fixed-order localization in
\cite[Sections~3.1--3.2]{Kook26}, we combine the analytic jet bounds
with a single Gaussian event that is uniform over the path and the
logarithmic truncation order. We retain the normalization
\eqref{eq:normalization}.

\subsection{A uniform path event}

At the base point, let $k_i$ be the unit row directions in \eqref{eq:parsevalrows}. Then
\begin{equation}
 \zeta_{0,i}=k_i^\transpose h.
 \label{eq:rhoGaussian}
\end{equation}
For $\eps\in(0,1/4)$, define
\begin{equation}
 \Lambda_\eps=1+\log\frac{2(m+d+1)}{\eps}.
 \label{eq:Lambdaeps}
\end{equation}

\begin{lemma}[Initial Gaussian event]\label{lem:initialevent}
There is a universal $C>0$ such that the event
\begin{equation}
 \cE_0=
 \left\{
 \norm{h}\le C\sqrt{d\Lambda_\eps},
 \quad
 \norm{\zeta_0}_\infty\le C\sqrt{\Lambda_\eps}
 \right\}
 \label{eq:E0}
\end{equation}
has probability at least $1-\eps/8$.
\end{lemma}

\begin{proof}
Gaussian concentration \cite{Borell75,Ledoux01} gives
\[
 \Pp\!\left(
 \norm h>\sqrt d+\sqrt{2\log(16/\eps)}
 \right)\le\frac{\eps}{16}.
\]
Each $k_i^\transpose h$ is standard normal, so the normal tail bound
and a union bound give
\[
 \Pp\!\left(
 \max_{1\le i\le m}\abs{k_i^\transpose h}
 >\sqrt{2\log(32m/\eps)}
 \right)\le\frac{\eps}{16}.
\]
Both displayed thresholds are bounded by those in
\eqref{eq:E0} for a sufficiently large universal $C$, since
$\Lambda_\eps\ge1$. The union bound proves the claim.
\end{proof}

\begin{lemma}[Identification on a real proposal segment]
\label{lem:realAnalyticIdentification}
Work under \eqref{eq:normalization}, let $h\ne0$,
$\lambda=\norm{h}$, $e=h/\lambda$, and $\tau=t\lambda$. On
$\cE_0$, if
\[
 0\le t\le r/\sqrt d,
 \qquad
 r\le c\rho_p\Lambda_\eps^{-1/2},
\]
then, after decreasing the universal constant $c>0$,
\[
 0\le\tau<\min\{\rho_p,1/2\}.
\]
Every slack along $x+\tau e$ is positive, the physical horizontal
Lewis frame exists on this segment, and it is the restriction of the
holomorphic solution in Lemma~\ref{lem:analyticflow}. In particular,
\begin{equation}
 C^{-1}w_{x,i}\le w_{x+th,i}\le Cw_{x,i},
 \qquad 1\le i\le m.
 \label{eq:realPathWeightComparison}
\end{equation}
\end{lemma}

\begin{proof}
On $\cE_0$,
$0\le\tau\le Cr\sqrt{\Lambda_\eps}$, so the stated radius condition
places the segment in $[0,\rho_p)\cap[0,1/2)$. By
\eqref{eq:sinitial}, $\abs{(A_xe)_i}\le1$, and hence
\[
 (A(x+\tau e)-b)_i
 =(Ax-b)_i\bigl(1+\tau(A_xe)_i\bigr)>0.
\]
Thus the segment lies in $\operatorname{int}K$. The physical
horizontal frame exists there and, by
Proposition~\ref{prop:physicalFrame}, satisfies the same real ODE
and initial condition as the restriction of the holomorphic frame.
Local uniqueness identifies them near zero. If their maximal interval
of equality ended earlier, \eqref{eq:Uball} would keep the common
solution in the row-map neighborhood, and local uniqueness at the
endpoint would extend it, a contradiction. Lemma~\ref{lem:rowmap}
and \eqref{eq:physicalWeights} now give
\eqref{eq:realPathWeightComparison}.
\end{proof}

\begin{lemma}[Uniform path bounds]\label{lem:pathbounds}
There are universal constants $c,C>0$ with the following property. Let
\begin{equation}
 \Gamma_\eps=C_\Gamma p^2\Lambda_\eps^2.
 \label{eq:Gammaeps}
\end{equation}
Here $C_\Gamma\ge1$ is a fixed universal constant chosen large
enough for all estimates in this lemma.
If
\begin{equation}
 r\le c\rho_p\Lambda_\eps^{-1/2},
 \label{eq:rpathcondition}
\end{equation}
then, on $\cE_0$, for every $0\le t\le\eta=r/\sqrt d$, the point $x_t=x+th$ belongs to $\operatorname{int}K$ and
\begin{align}
 \norm{\zeta_t}_\infty+\norm{s_t}_\infty&\le\Gamma_\eps,
 \label{eq:pathrho}\\
 \norm{u_t}+\norm{v_t}+\norm{q_t}+\norm{q_t'}+\norm{v_t'}
 &\le\Gamma_\eps\sqrt d,
 \label{eq:pathbasic}\\
 \abs{E_1'(t)}&\le\Gamma_\eps d,
 \label{eq:pathE}\\
 \norm{h}_{g_0(x_t)}&\le\Gamma_\eps\sqrt d.
 \label{eq:pathmetric}
\end{align}
For every integer $k\ge1$,
\begin{equation}
 \norm{N_t^{(k)}}_{\op}
 \le Cp^2 k!\rho_p^{-k}\Gamma_\eps^k d^{k/2}.
 \label{eq:pathNall}
\end{equation}
Consequently, for $j\ge2$,
\begin{align}
 \abs{G_j(t)}
 &\le Cp^2\Gamma_\eps^2(j-2)!
 \left(\frac{\Gamma_\eps}{\rho_p}\right)^{j-2}d^{j/2},
 \label{eq:pathGj}\\
 \abs{H_j(t)}
 &\le Cp^2\Gamma_\eps^2(j-1)!
 \left(\frac{\Gamma_\eps}{\rho_p}\right)^{j-1}d^{(j+1)/2}.
 \label{eq:pathHj}
\end{align}
\end{lemma}

\begin{proof}
If $h=0$, every assertion is immediate. Otherwise put
$\lambda=\norm{h}$, $e=h/\lambda$, and $\tau=t\lambda$.
Lemma~\ref{lem:realAnalyticIdentification} gives
$x_t\in\operatorname{int}K$ and
\eqref{eq:realPathWeightComparison} throughout the required
interval. We use this identification below.

For the original direction $h$, the exact formula is
\begin{equation}
 s_{t,i}=\frac{s_{0,i}}{1+ts_{0,i}},
 \qquad s_{0,i}=w_{0,i}^{1/p}\zeta_{0,i}.
 \label{eq:slackpath}
\end{equation}
The event $\cE_0$ gives $\norm{s_0}_\infty\le C\sqrt{\Lambda_\eps}$. Decreasing $c$ if necessary, $\abs{ts_{0,i}}\le1/2$. Hence $\norm{s_t}_\infty\le C\sqrt{\Lambda_\eps}$. Also,
\begin{equation}
 \zeta_{t,i}=\zeta_{0,i}
 \left(\frac{w_{t,i}}{w_{0,i}}\right)^{-1/p}
 (1+ts_{0,i})^{-1},
 \label{eq:rhoprop}
\end{equation}
so \eqref{eq:realPathWeightComparison} proves \eqref{eq:pathrho}.

The identities
\begin{equation}
 u_{t,i}=\sqrt{w_{t,i}}\zeta_{t,i},
 \qquad v_t=W_t^{1/p}u_t,
 \qquad q_{t,i}=\sqrt{w_{t,i}}\zeta_{t,i}^2
 \label{eq:uvqrho}
\end{equation}
and $\sum_iw_{t,i}=d$ imply
\begin{equation}
 \norm{u_t}+\norm{v_t}+\norm{q_t}
 \le C\Gamma_\eps\sqrt d.
 \label{eq:uvqbound}
\end{equation}
The derivative formulas are
\begin{align}
 v_t'&=-s_t\had v_t-\frac12(N_tv_t)\had s_t,
 \label{eq:vprime}\\
 q_t'&=\left(\left(2\beta-\frac12\right)\ell_t'-2s_t\right)\had q_t,
 \qquad \ell_t'=-W_t^{-1/2}N_tv_t.
 \label{eq:qprime}
\end{align}
At the real point $x_t$, $0\preceq N_t\preceq p\Id_m$, so
\begin{equation}
 \norm{N_tv_t}\le p\norm{v_t}.
 \label{eq:Nvpath}
\end{equation}
For $q_t'$, use the exact cancellation
\begin{equation}
 \ell_t'\had q_t
 =-\bigl(W_t^{-1/2}N_tv_t\bigr)
 \had\bigl(W_t^{-1/2}u_t^{\had2}\bigr)
 =-(N_tv_t)\had\zeta_t^{\had2}.
 \label{eq:ellq}
\end{equation}
Indeed, with $R=C\sqrt{\Lambda_\eps}$ denoting a common bound for
$\norm{s_t}_\infty$ and $\norm{\zeta_t}_\infty$,
\[
 \norm{v_t'}\le R\left(\norm{v_t}+\tfrac12\norm{N_tv_t}\right)
 \le CpR^2\sqrt d,
\]
and the sharper identity \eqref{eq:ellq} gives
\[
 \norm{q_t'}
 \le C\left(\norm{\zeta_t}_\infty^2\norm{N_tv_t}
       +\norm{s_t}_\infty\norm{q_t}\right)
 \le CpR^3\sqrt d.
\]
Together with \eqref{eq:uvqbound}, these inequalities prove
\eqref{eq:pathbasic} after fixing the stated universal
$C_\Gamma$ sufficiently large.

Differentiate $E_1(t)=\sum_iw_{t,i}^\alpha s_{t,i}^3$:
\begin{equation}
 E_1'(t)=\alpha\sum_iw_{t,i}^\alpha\ell_{t,i}'s_{t,i}^3
 -3\sum_iw_{t,i}^\alpha s_{t,i}^4.
 \label{eq:Eprimeexact}
\end{equation}
The second sum is at most $3\norm{s_t}_\infty^2\norm{u_t}^2$. For the first, Cauchy--Schwarz and $W_t^{1/2}\ell_t'=-N_tv_t$ give
\begin{align*}
 \sum_iw_i^\alpha\abs{\ell_i'}\abs{s_i}^3
 &\le\left(\sum_iw_i^\alpha s_i^2\right)^{1/2}
 \left(\sum_iw_i^\alpha(\ell_i')^2s_i^4\right)^{1/2}\\
 &\le\norm{u_t}\norm{\zeta_t}_\infty\norm{s_t}_\infty
 \norm{N_tv_t}.
\end{align*}
This proves \eqref{eq:pathE}. Since $F(t)=\norm{u_t}^2=\norm{h}_{g_0(x_t)}^2$, \eqref{eq:pathmetric} follows as well.

Apply \eqref{eq:diagNjet} at $x_t$ in the direction $h$ and use \eqref{eq:pathmetric} to obtain \eqref{eq:pathNall}. Inserting \eqref{eq:pathbasic} and \eqref{eq:pathNall} into the definitions of $G_j$ and $H_j$ proves \eqref{eq:pathGj} and \eqref{eq:pathHj}.
More explicitly, for $j\ge3$ the two factors outside
$N_t^{(j-2)}$ in $G_j$ contribute at most
$\Gamma_\eps^2d$, and hence
\[
 \abs{G_j(t)}
 \le Cp^2\Gamma_\eps^2d\,
 (j-2)!\rho_p^{-(j-2)}
 \Gamma_\eps^{j-2}d^{(j-2)/2},
\]
which is \eqref{eq:pathGj}.  When $j=2$, the same conclusion follows
directly from $\norm{N_t}_{\op}\le p$, so no zeroth-order use of
\eqref{eq:pathNall} is required.  Similarly, the two exterior
factors in $H_j$ contribute $\Gamma_\eps^2d$, while
$N_t^{(j-1)}$ contributes
\[
 Cp^2(j-1)!\rho_p^{-(j-1)}
 \Gamma_\eps^{j-1}d^{(j-1)/2}.
\]
Their product is exactly the right side of \eqref{eq:pathHj}; in
particular its dimension power is
$d\,d^{(j-1)/2}=d^{(j+1)/2}$.
\end{proof}

\subsection{Summing the expansion}

The parameter choice below keeps localization valid, bounds each
low-order contribution by $O(\eps)$, absorbs the terminal $\sqrt d$
factor, and retains an inverse-polylogarithmic radius.

\begin{lemma}[Compatible choice of truncation parameters]
\label{lem:constants}
Fix once and for all the constants in all preceding estimates. Let
$c_{\rm path}\in(0,1]$ be a valid constant in
\eqref{eq:rpathcondition}, and write
$\rho_p=c_\rho p^{-3}$ with $c_\rho\in(0,1]$. There are universal
constants
\[
 C_{\rm cl},C_2,C_Q,C_J,C_S,A_0,B_0\ge1,
 \qquad c_S,c_{\rm low}>0,
\]
which can be fixed in the displayed order
\begin{equation}
 C_{\rm cl}\ ;\ C_2\ ;\ C_Q\ ;\ C_J\ ;\ C_S\ ;\
 c_S\ ;\ (A_0,B_0,c_{\rm low}),
 \label{eq:constantOrder}
\end{equation}
such that the following holds. Define
\begin{align}
 Q_\eps&=C_Qp^4\Gamma_\eps^4,
 \label{eq:Qeps}\\
 J&=\left\lceil
 C_J\log\!\left(\frac{2(d+1)Q_\eps}{\eps}\right)
 \right\rceil,
 \label{eq:Jchoice}\\
 r_{\rm cand}
 &=c_S\eps^3\rho_pQ_\eps^{-1}J^{-C_S}.
 \label{eq:rcandidate}
\end{align}
For every $0<r\le r_{\rm cand}$, if
\[
 \theta=\frac{r\Gamma_\eps}{\rho_p},
 \qquad
 \xi=(C_{\rm cl}J)^{C_2}\frac r{\rho_p},
\]
then
\begin{align}
 r&\le c_{\rm path}\rho_p\Lambda_\eps^{-1/2},
 &\theta&\le J^{-4},
 &\xi&\le\frac14,
 \label{eq:constantSeries}\\
 C_{\rm cl}p^4r^2\eps^{-2}&\le\frac\eps8,
 &\frac12\Gamma_\eps r^2&\le\frac\eps{20},
 &2p^2\rho_p\xi^2&\le\frac{\eps^2}{160},
 \label{eq:constantLow}\\
 C_{\rm cl}p^2\Gamma_\eps^2r^2&\le\frac\eps{40},
 &C_{\rm cl}p^2\Gamma_\eps^2r\sqrt d\,\theta^J
 &\le\frac\eps{40}.
 \label{eq:constantTerminal}
\end{align}
Moreover,
\begin{equation}
 r_{\rm cand}\ge
 c_{\rm low}\eps^3p^{-A_0}\Lambda_\eps^{-B_0}.
 \label{eq:candidateLower}
\end{equation}
\end{lemma}

\begin{proof}
Choose $C_{\rm cl}$ to dominate every universal coefficient already
fixed in the cubic, jet, Wick and path estimates used below, including
the fixed base constant inside $(Cj)^{Cj}$ in
\eqref{eq:Hbase}. Next choose $C_2$ so that, for all integers
$J\ge j\ge2$,
\begin{equation}
 \frac{(Cj)^{Cj}}{j!}(1+J)^{j/2}
 \le(C_{\rm cl}J)^{C_2j},
 \label{eq:namedStirling}
\end{equation}
where the $C$ on the left is the now fixed constant from
\eqref{eq:Hbase}. This is possible by Stirling's inequality.
Choose $C_Q$ so large that
\begin{equation}
 Q_\eps\ge C_{\rm cl}
 \max\{1,p^4,p^2\Gamma_\eps^2,\Gamma_\eps\}
 \label{eq:Qdominates}
\end{equation}
for every $p\ge4$ and $\Lambda_\eps\ge1$. Next choose $C_J\ge1$ for the
terminal estimate below, choose $C_S\ge C_2+4$, and only then choose
$c_S>0$ sufficiently small in terms of these fixed universal
constants. This is the order in \eqref{eq:constantOrder}.

Since $\Gamma_\eps=C_\Gamma p^2\Lambda_\eps^2$,
\[
 Q_\eps=C_QC_\Gamma^4p^{12}\Lambda_\eps^8.
\]
In particular, with
\[
 X:=\frac{2(d+1)Q_\eps}{\eps},
\]
one has $X\ge16$.  Thus the ceiling in \eqref{eq:Jchoice} defines an
integer $J\ge3$, and
\begin{equation}
 C_J\log X\le J\le C_J\log X+1\le2C_J\log X.
 \label{eq:JceilingBounds}
\end{equation}
For $0<r\le r_{\rm cand}$,
\begin{align*}
 \frac{r\Lambda_\eps^{1/2}}{\rho_p}
 &\le c_S\eps^3Q_\eps^{-1}J^{-C_S}
 \Lambda_\eps^{1/2}
 \le Cc_S\Lambda_\eps^{-15/2},\\
 \theta
 &\le c_S\eps^3\frac{\Gamma_\eps}{Q_\eps}J^{-C_S}
 \le c_SJ^{-C_S},\\
 \xi
 &\le c_SC_{\rm cl}^{C_2}\eps^3Q_\eps^{-1}
 J^{C_2-C_S}.
\end{align*}
Thus a sufficiently small $c_S$, chosen after $C_S$, gives all three
inequalities in \eqref{eq:constantSeries}; in particular
$r\le\rho_p\le1$. The remaining low-order estimates follow from
\begin{align*}
 C_{\rm cl}p^4r^2\eps^{-2}
 &\le C_{\rm cl}c_S^2\eps^4
 \frac{p^4\rho_p^2}{Q_\eps^2},\\
 \frac12\Gamma_\eps r^2
 &\le\frac12c_S^2\eps^6
 \frac{\rho_p^2\Gamma_\eps}{Q_\eps^2},\\
 2p^2\rho_p\xi^2
 &\le
 2c_S^2C_{\rm cl}^{2C_2}\eps^6
 \frac{p^2\rho_p}{Q_\eps^2}J^{2C_2-2C_S},\\
 C_{\rm cl}p^2\Gamma_\eps^2r^2
 &\le C_{\rm cl}c_S^2\eps^6
 \frac{p^2\rho_p^2\Gamma_\eps^2}{Q_\eps^2}.
\end{align*}
Here
$p^4\rho_p^2\le1$, $p^2\rho_p\le1$, and every remaining displayed
ratio is at most one by \eqref{eq:Qdominates}. Since
$\eps^4\le\eps$, $\eps^6\le\eps^2$, and $\eps^2\le\eps$ for
$\eps\in(0,1/4)$, decreasing $c_S$ once proves
\eqref{eq:constantLow} and the first inequality in
\eqref{eq:constantTerminal}.

For the last terminal inequality, retain the preceding value of
$X$. Then
$Q_\eps\sqrt d\le\eps X/2$, and
\eqref{eq:Qdominates}, $r\le1$, and $\theta\le J^{-4}$ give
\[
 C_{\rm cl}p^2\Gamma_\eps^2r\sqrt d\,\theta^J
 \le Q_\eps\sqrt d\,J^{-4J}.
\]
Choose $C_J$ sufficiently large that the integer in
\eqref{eq:Jchoice} satisfies $J\ge e$; the inequality
$J\ge C_J\log X$ is automatic from
\eqref{eq:JceilingBounds}. Then
\[
 J^{-4J}\le e^{-4J}\le X^{-4C_J},
 \qquad
 Q_\eps\sqrt d\,J^{-4J}
 \le\frac\eps2X^{1-4C_J}\le\frac\eps{40}.
\]
This choice of $C_J$ precedes the choices of $C_S$ and $c_S$, as
required.

Finally,
\[
 \log X
 \le C\bigl(\Lambda_\eps+\log p+\log\Lambda_\eps\bigr)
 \le Cp\Lambda_\eps,
\]
and hence \eqref{eq:JceilingBounds} gives
$J\le Cp\Lambda_\eps$. Since
$Q_\eps\le Cp^{12}\Lambda_\eps^8$ and
$\rho_p=c_\rho p^{-3}$, substitution in
\eqref{eq:rcandidate} gives
\[
 r_{\rm cand}
 \ge c_{\rm low}\eps^3
 p^{-(15+C_S)}\Lambda_\eps^{-(8+C_S)}.
\]
Thus one may set $A_0=15+C_S$ and $B_0=8+C_S$, with
$c_{\rm low}>0$ depending only on the constants already fixed in
\eqref{eq:constantOrder}.
\end{proof}

\begin{proof}[Proof of Theorem~\ref{thm:asc}]
Fix $\eps\in(0,1/4)$ and the quantities in
Lemma~\ref{lem:constants}. After decreasing $c_{\rm asc}$ and
increasing $C_{\rm asc}$ in \eqref{eq:reps},
\eqref{eq:candidateLower} gives
$\underline r_\eps\le r_{\rm cand}$. Fix an arbitrary
$0<r\le\underline r_\eps$, set $\eta=r/\sqrt d$, and put
\begin{equation}
 \theta=\frac{r\Gamma_\eps}{\rho_p}.
 \label{eq:theta}
\end{equation}
Lemma~\ref{lem:constants} supplies all smallness conditions below,
including \eqref{eq:rpathcondition}.
By \eqref{eq:Fprime}, Lemma~\ref{lem:selective-expansion}, and one integration of $E_1$,
\begin{align}
 \frac12\abs{F(\eta)-F(0)}
 \le{}&\eta\abs{E_1(0)+\beta H_1(0)}
 +\int_0^\eta(\eta-t)\abs{E_1'(t)}\dd t
 \notag\\
 &+\beta\sum_{j=2}^J\frac{\eta^j}{j!}\abs{H_j(0)}
 \notag\\
 &+\beta\sum_{j=2}^J\int_0^\eta
 \frac{(\eta-t)^{j-1}}{(j-1)!}\abs{G_j(t)}\dd t
 \notag\\
 &+\beta\int_0^\eta\frac{(\eta-t)^J}{J!}
 \bigl(\abs{G_{J+1}(t)}+\abs{H_{J+1}(t)}\bigr)\dd t.
 \label{eq:masterbound}
\end{align}

By Lemma~\ref{lem:cubic} and Chebyshev's inequality,
\begin{equation}
 \Pp\!\left(
 \eta\abs{E_1(0)+\beta H_1(0)}>\frac\eps{20}
 \right)
 \le C_{\rm cl}p^4\frac{r^2}{\eps^2}\le\frac\eps8.
 \label{eq:cubicprob}
\end{equation}
The final inequality is part of \eqref{eq:constantLow}.

On $\cE_0$, Lemma~\ref{lem:pathbounds} gives
\begin{equation}
 \int_0^\eta(\eta-t)\abs{E_1'(t)}\dd t
 \le\frac12\eta^2\Gamma_\eps d
 =\frac12\Gamma_\eps r^2\le\frac\eps{20}.
 \label{eq:Eterm}
\end{equation}

Let
\[
 S_J=\sum_{j=2}^J\frac{\eta^j}{j!}\abs{H_j(0)}.
\]
Theorem~\ref{thm:Hbase} yields
\begin{align}
 \E S_J
 &\le p^2\rho_p\sum_{j=2}^J
 \frac{(Cj)^{Cj}}{j!}
 \left(\frac r{\rho_p}\right)^j
 \left(1+\frac jd\right)^{j/2}.
 \label{eq:ESK}
\end{align}
For $j\le J$ and $d\ge1$, the last factor is at most
$(1+J)^{j/2}$. The named Stirling estimate
\eqref{eq:namedStirling} therefore bounds the factor multiplying
$(r/\rho_p)^j$ by $(C_{\rm cl}J)^{C_2j}$. Put
$\xi=(C_{\rm cl}J)^{C_2}r/\rho_p$. By
\eqref{eq:constantSeries}, $\xi\le1/4$. Consequently,
\[
 \E S_J\le p^2\rho_p\sum_{j=2}^{\infty}\xi^j
 \le 2p^2\rho_p\xi^2
 \le\frac{\eps^2}{160}
\]
by \eqref{eq:constantLow}. Thus
\begin{equation}
 \E S_J\le\frac{\eps^2}{160}.
 \label{eq:ESKsmall}
\end{equation}
Markov's inequality implies
\begin{equation}
 \Pp\!\left(S_J>\frac\eps{20}\right)\le\frac\eps8.
 \label{eq:SKprob}
\end{equation}

On $\cE_0$, \eqref{eq:pathGj} and
$\int_0^\eta(\eta-t)^{j-1}\dd t=\eta^j/j$ give
\begin{align}
 &\sum_{j=2}^J\int_0^\eta
 \frac{(\eta-t)^{j-1}}{(j-1)!}\abs{G_j(t)}\dd t\notag\\
 &\qquad\le C_{\rm cl}p^2\Gamma_\eps^2r^2
 \sum_{j=2}^J\frac{\theta^{j-2}}{j(j-1)}
 \le C_{\rm cl}p^2\Gamma_\eps^2r^2
 \le\frac\eps{20}.
 \label{eq:Gseries}
\end{align}

The terminal $G_{J+1}$ term is at most
\begin{equation}
 C_{\rm cl}p^2\Gamma_\eps^2r^2
 \frac{\theta^{J-1}}{J(J+1)}.
 \label{eq:Gterminal}
\end{equation}
Indeed, integrating the order-$J+1$ instance of
\eqref{eq:pathGj} gives
\[
 \frac{\eta^{J+1}}{(J+1)!}\,
 Cp^2\Gamma_\eps^2(J-1)!
 \left(\frac{\Gamma_\eps}{\rho_p}\right)^{J-1}
 d^{(J+1)/2},
\]
and $\eta=r/\sqrt d$ turns this expression into
\eqref{eq:Gterminal}; all powers of $d$ cancel.
The first inequality in \eqref{eq:constantTerminal}, together with
$\theta\le J^{-4}\le1$ and $J\ge3$, therefore gives
\begin{equation}
 C_{\rm cl}p^2\Gamma_\eps^2r^2
 \frac{\theta^{J-1}}{J(J+1)}
 \le C_{\rm cl}p^2\Gamma_\eps^2r^2
 \le\frac\eps{40}.
 \label{eq:GterminalSmall}
\end{equation}
For the terminal $H_{J+1}$ term, \eqref{eq:pathHj} gives
\begin{equation}
 \int_0^\eta\frac{(\eta-t)^J}{J!}\abs{H_{J+1}(t)}\dd t
 \le C_{\rm cl}p^2\Gamma_\eps^2
 \frac{r\sqrt d}{J+1}\theta^J.
 \label{eq:Hterminal}
\end{equation}
Here the corresponding unsimplified bound is
\[
 \frac{\eta^{J+1}}{(J+1)!}\,
 Cp^2\Gamma_\eps^2J!
 \left(\frac{\Gamma_\eps}{\rho_p}\right)^J
 d^{(J+2)/2}.
\]
Thus exactly one factor $\sqrt d$ remains:
$d^{(J+2)/2}d^{-(J+1)/2}=\sqrt d$. The terminal estimate therefore
retains exactly this factor. The logarithmic choice of $J$ absorbs it
through $\theta^J$, as quantified in \eqref{eq:constantTerminal}.
The second inequality in \eqref{eq:constantTerminal} gives
\begin{equation}
 C_{\rm cl}p^2\Gamma_\eps^2r\sqrt d\,\theta^J
 \le\frac\eps{40}.
 \label{eq:terminalkill}
\end{equation}
Since $(J+1)^{-1}\le1$, \eqref{eq:Hterminal} and
\eqref{eq:terminalkill} show separately that the terminal
$H_{J+1}$ term is at most $\eps/40$.  Together with
\eqref{eq:GterminalSmall}, this proves that the entire terminal
remainder is at most $\eps/20$.

The failures of $\cE_0$, \eqref{eq:cubicprob}, and
\eqref{eq:SKprob} have total probability at most
\[
 \frac\eps8+\frac\eps8+\frac\eps8
 =\frac{3\eps}{8}\le\eps.
\]
On their complement, each of the five contributions on the right side
of \eqref{eq:masterbound} is at most $\eps/20$, and $0<\beta<1$. Hence
$\frac12\abs{F(\eta)-F(0)}\le\eps/4$, and therefore
\begin{equation}
 \abs{F(\eta)-F(0)}\le\frac\eps2\le2\eps.
 \label{eq:FASC}
\end{equation}
If $z=x+\eta h$, then
\[
 \norm{z-x}_{g_0(z)}^2-\norm{z-x}_{g_0(x)}^2
 =\eta^2(F(\eta)-F(0)).
\]
Since $\eta^2=r^2/d$, \eqref{eq:FASC} gives the stronger two-sided
estimate
\[
 \abs{\norm{z-x}_{g_0(z)}^2-\norm{z-x}_{g_0(x)}^2}
 \le\frac{\eps r^2}{2d}
 \le\frac{2\eps r^2}{d}.
\]
Together with $z\in\operatorname{int}K$, which is guaranteed by the
path event, this implies the event in \eqref{eq:quantASCdef}; in
particular, it implies the one-sided ASC condition used by the mixing
framework. The base point was arbitrary, and
Lemma~\ref{lem:affineCovariance} transfers both the discrepancy event
and its probability back from the normalization
\eqref{eq:normalization}. The same argument holds with any
$0<s\le r$ in place of $r$. Hence the estimate is uniform in $x$ and
$s$, and the success probability is at least
$1-3\eps/8\ge1-\eps$, proving Theorem~\ref{thm:asc}.
\end{proof}

\begin{remark}[The fourth-order loss]
At order four, a pathwise Cauchy--Schwarz estimate gives
\[
 \abs{H_4(t)}
 \le\norm{q_t}\norm{N_t^{(3)}}_{\op}\norm{v_t}
 \lesssim d^{5/2}.
\]
Theorem~\ref{thm:Hbase} instead gives
$\norm{H_4(0)}_{L^2}\le p^{O(1)}d^2$. The expansion uses the
base-point estimate at each retained order and postpones the pathwise
bound to order $J$, where $\theta^J$ absorbs the residual
$d^{1/2}$ factor.
\end{remark}

\section{From average self-concordance to mixing}\label{sec:mixing}

We now rescale the baseline ASC radius into a dimension-free modulus
and insert it, together with the deterministic geometry of Section
\ref{sec:prelim}, into the abstract mixing reduction.

\begin{lemma}[Dimension-free ASC modulus after scaling]
\label{lem:scaledASC}
With $p=p_m$ and $L_{m,d}$ from \eqref{eq:pchoice}--\eqref{eq:Lchoice},
the family of metrics $g_\star=L_{m,d}g_0$ has an ASC modulus
$\mathfrak a_\star:(0,1)\to(0,\infty)$ independent of $m$ and $d$.
One may take
\[
 \mathfrak a_\star(\eps)
 =c_\star
 \begin{cases}
 \eps^3(1+\log(1/\eps))^{-C_{\rm asc}},
 &0<\eps\le\eps_0,\\
 1,&\eps_0<\eps<1.
 \end{cases}
\]
\end{lemma}

\begin{proof}
For $0<\eps\le\eps_0$, put
$B=1+\log(2(m+d+1)/\eps_0)\ge1$. Then
\[
 \frac{\underline r_\eps}{\underline r_{\eps_0}}
 =\left(\frac{\eps}{\eps_0}\right)^3
 \left(\frac{B}{B+\log(\eps_0/\eps)}\right)^{C_{\rm asc}}
 \ge c\eps^3(1+\log(1/\eps))^{-C_{\rm asc}}.
\]
Proposition~\ref{prop:scaling}, Theorem~\ref{thm:asc}, and
$\widehat r_0\le\underline r_{\eps_0}$ give
\[
 R_\eps(g_\star)
 \ge L_{m,d}^{1/2}\underline r_\eps
 \ge \frac{\underline r_\eps}{\widehat r_0}
 \ge \frac{\underline r_\eps}{\underline r_{\eps_0}}
 \ge c\eps^3(1+\log(1/\eps))^{-C_{\rm asc}},
\]
which proves the first case. For
$\eps>\eps_0$, the event and the required success probability are
weaker than at $\eps_0$, and explicitly
\[
 R_\eps(g_\star)\ge R_{\eps_0}(g_\star)
 \ge L_{m,d}^{1/2}\underline r_{\eps_0}\ge1.
\]
\end{proof}

\begin{proof}[Proof of Theorem~\ref{thm:main}]
The choice $p=p_m$ satisfies the hypothesis of
Proposition~\ref{prop:knownregularity}. Thus
$H_{p_m}=\lambda_mg_0$ is strongly self-concordant and lower-trace
self-concordant, with symmetry parameter at most
$e\lambda_md$. Since
\[
 g_\star=L_{m,d}g_0=\widehat r_0^{-2}H_{p_m}
 \quad\text{and}\quad \widehat r_0^{-2}\ge1,
\]
Proposition~\ref{prop:scaling} shows that $g_\star$ has the same two
differential regularity properties and symmetry parameter
\begin{equation}
 \bar\nu(g_\star)\le e\lambda_md\widehat r_0^{-2}
 =edL_{m,d}.
 \label{eq:finalSymmetry}
\end{equation}
Lemma~\ref{lem:scaledASC} supplies a dimension-free ASC modulus.
Hence the constant $a_0$ in Proposition~\ref{prop:framework} is
universal, while \eqref{eq:finalSymmetry} gives
\[
 d\bar\nu(g_\star)\le e d^2L_{m,d}.
\]
Proposition~\ref{prop:framework} therefore yields an $L^2$ spectral
gap of order $(d^2L_{m,d})^{-1}$ and the iteration bound
\eqref{eq:mainmix}. Its proposal radius depends only on universal
constants.

Finally, $p_m=O(1+\log m)$,
$\lambda_m=p_m^{O(1)}$, and Theorem~\ref{thm:asc} with fixed
$\eps_0$ gives
$\widehat r_0^{-1}=\polylog(m,d)$. Hence
$L_{m,d}=\polylog(m,d)$.
\end{proof}

\begin{corollary}[Uniform sampling]\label{cor:uniform}
Let $\pi$ be the uniform distribution on a bounded full-dimensional
polytope, and suppose that
$\mu_0\ll\pi$ satisfies
\[
 \left\Vert\frac{\dd\mu_0}{\dd\pi}\right\Vert_\infty\le M
\]
for some $M\ge1$. For every $\delta\in(0,1)$, run the Dikin walk with
the scaled metric $g_\star$ for
\begin{equation}
 T\ge C d^2L_{m,d}\log\!\left(\frac{M}{\delta}\right).
 \label{eq:uniformWarmStart}
\end{equation}
Then its law satisfies $\chi^2(\mu_T\Vert\pi)\le\delta$.
Thus its warm-start iteration complexity is
\[
 \widetilde O\!\left(d^2\log\frac{M}{\delta}\right).
\]
\end{corollary}

\begin{proof}
The warmness assumption gives
\[
 1+\chi^2(\mu_0\Vert\pi)
 =\int\left(\frac{\dd\mu_0}{\dd\pi}\right)^2\dd\pi
 \le M.
\]
Apply Theorem~\ref{thm:main} with $\vartheta=0$.
\end{proof}

\begin{remark}[Logarithmic factors]
The powers of $p$, $\log m$, $\log d$ and $\eps^{-1}$ have not been optimized. The analytic argument gives the explicit local scale $\rho_p\asymp p^{-3}$; later estimates deliberately enlarge fixed polynomial powers because they do not affect the dimension exponent.
\end{remark}

\appendix

\section{A complex analytic ODE lemma}\label{app:analyticODE}

\begin{lemma}\label{lem:abstractODE}
Let $(X,\norm\cdot)$ be a finite-dimensional complex Banach space. Suppose $F:B(x_0,\delta)\times D(0,\tau_0)\to X$ is holomorphic and
\[
 \sup_{B(x_0,\delta)\times D(0,\tau_0)}\norm{F(x,t)}\le M.
\]
Then the initial-value problem
\[
 x'(t)=F(x(t),t),\qquad x(0)=x_0,
\]
has a unique holomorphic solution on
\begin{equation}
 \abs t<2r_0,
 \qquad r_0=c\min\{\tau_0,\delta/M\},
 \label{eq:abstractradius}
\end{equation}
for a sufficiently small universal $c>0$, and $\norm{x(t)-x_0}\le\delta/2$ there. If $H$ is holomorphic on $B(x_0,\delta)$ and
$\sup_{B(x_0,\delta/2)}\norm{H-H(x_0)}\le B$, then
\begin{equation}
 \norm{\frac{\dd^k}{\dd t^k}H(x(t))\big|_{t=0}}
 \le Bk!r_0^{-k},
 \qquad k\ge1.
 \label{eq:abstractCauchy}
\end{equation}
\end{lemma}

\begin{proof}
For $x\in B(x_0,3\delta/4)$ and $\norm{h}=1$, the one-variable map
\[
 z\longmapsto F(x+zh,t),\qquad \abs z<\delta/4,
\]
takes values in the radius-$M$ ball of $X$. The Banach-valued Cauchy
formula therefore gives
\[
 \norm{D_xF(x,t)[h]}\le4M/\delta,
\]
and hence $\norm{D_xF(x,t)}\le4M/\delta$. This argument uses no
equivalence between the chosen norm and a Euclidean norm, so its
constant is independent of $\dim X$.

Let $\cY$ be the complete metric space of holomorphic functions on
$D(0,2r_0)$, continuous on its closure, satisfying
$\sup_{\abs t\le 2r_0}\norm{x(t)-x_0}\le\delta/2$. Define, with the
integral taken along the line segment from $0$ to $t$,
\[
 (\cR x)(t)=x_0+\int_0^tF(x(s),s)\dd s.
\]
Choose the universal constant in \eqref{eq:abstractradius} so that
$2r_0<\tau_0$, $2r_0M\le\delta/2$, and
$8r_0M/\delta\le1/2$. The preceding derivative bound, integrated
along the segment between $x(s)$ and $\widetilde x(s)$, shows that
\[
 \sup_{\abs t\le2r_0}
 \norm{\cR x(t)-\cR\widetilde x(t)}
 \le\frac12\sup_{\abs t\le2r_0}
 \norm{x(t)-\widetilde x(t)}.
\]
Thus $\cR$ is invariant and contractive. Its fixed point is the
unique holomorphic solution staying in the ball; ordinary local
uniqueness then gives uniqueness for the initial-value problem on the
whole disk. On $\abs t\le r_0$, the sharper estimate
$\norm{x(t)-x_0}\le r_0M\le\delta/4$ places the Cauchy circle inside
$B(x_0,\delta/2)$. Applying the Banach-valued Cauchy formula to
$H(x(t))-H(x_0)$ proves \eqref{eq:abstractCauchy}.
\end{proof}

\section{Horizontal frames}\label{app:frame}

\begin{lemma}[Horizontal gauge]\label{lem:horizontalgauge}
Let $E_t\subset\R^m$ be a smooth path of $d$-dimensional subspaces, and let $\widetilde U_t$ be a smooth orthonormal frame for $E_t$. There is a smooth orthogonal matrix $O_t\in\R^{d\times d}$ such that $U_t=\widetilde U_tO_t$ satisfies
\[
 U_t^\transpose U_t=\Id_d,
 \qquad U_t^\transpose U_t'=0.
\]
If $P_t=U_tU_t^\transpose$ and a full-rank matrix $\widehat B_t$ with column space $E_t$ satisfies $\widehat B_t'=Z_t\widehat B_t$, then
\[
 U_t'=(\Id-P_t)Z_tU_t.
\]
\end{lemma}

\begin{proof}
Set $S_t=\widetilde U_t^\transpose\widetilde U_t'$. Differentiating
$\widetilde U_t^\transpose\widetilde U_t=\Id$ shows that $S_t$ is
skew-symmetric. Solve
$O_t'=-S_tO_t$, $O_0=\Id$. Then
\[
 (O_t^\transpose O_t)'=0,\qquad
 U_t^\transpose U_t'
 =O_t^\transpose S_tO_t+O_t^\transpose O_t'=0,
\]
so $O_t$ is orthogonal and $U_t$ is horizontal.

Write $\widehat B_t=U_tM_t$ with $M_t$ invertible. From
\[
 U_t'M_t+U_tM_t'=Z_tU_tM_t
\]
left multiplication by $\Id-P_t$ and right multiplication by
$M_t^{-1}$ give
\[
 (\Id-P_t)U_t'=(\Id-P_t)Z_tU_t.
\]
The horizontal condition gives
$P_tU_t'=U_t(U_t^\transpose U_t')=0$, proving the formula.
\end{proof}

\section{Gaussian tensor identities}\label{app:wick}

\begin{lemma}[Raw monomials and Wick tensors]\label{lem:rawWickAppendix}
For every symmetric $n$-tensor $T$ and $h\sim N(0,\Id_d)$,
\[
 T[h^{\otimes n}]
 =\sum_{r=0}^{\lfloor n/2\rfloor}
 \frac{n!}{2^rr!(n-2r)!}I_{n-2r}(\Tr^rT).
\]
Moreover, for symmetric tensors $S,T$ of orders $a,b$,
\[
 I_a(S)I_b(T)
 =\sum_{\ell=0}^{a\wedge b}
 \ell!\binom a\ell\binom b\ell
 I_{a+b-2\ell}(S\widetilde\otimes_\ell T),
\]
and
\[
 \E[I_a(S)I_b(T)]
 =\mathbf1_{\{a=b\}}a!\ip{S}{T}_{\HS}.
\]
Here $I_0(c)=c$, and a summand in the product formula is present only
for $0\le\ell\le\min\{a,b\}$.  In particular, an order-zero output
is a scalar and its Hilbert--Schmidt norm is its absolute value.
\end{lemma}

\begin{proof}
The product and isometry formulas are the standard identities for
multiple Wiener integrals \cite[Chapter~1]{Nualart06}.  The
symmetrized contraction in the product formula is the orthogonal
projection of the unsymmetrized contraction onto the symmetric
tensor subspace, so
\[
 \norm{S\widetilde\otimes_\ell T}_{\HS}
 \le\norm{S\otimes_\ell T}_{\HS}.
\]
This observation also applies when the output order is zero.

The raw-to-Wick formula follows by induction on
$n$.  Multiplication by $h_j=I_1(e_j)$ either raises the chaos order
by one or contracts $e_j$ with one existing index.  A term with
exactly $r$ contractions is indexed by a choice of $r$ disjoint
unordered pairs among the $n$ tensor slots.  Their number is
\[
 \binom n{2r}\frac{(2r)!}{2^rr!}
 =\frac{n!}{2^rr!(n-2r)!},
\]
which gives the asserted coefficient.  Distinct values of $r$
produce distinct chaos orders $n-2r$, and hence are orthogonal by
the isometry formula.  This last orthogonality statement concerns
the raw-to-Wick expansion itself; multiplication by an additional
Wick polynomial can make different contraction patterns land in the
same output chaos.
\end{proof}

\begin{lemma}[Gaussian radial moment]\label{lem:radialmoment}
For $h\sim N(0,\Id_d)$ and every integer $n\ge1$,
\begin{equation}
 \E\norm{h}^{2n}\le(C(d+n))^n.
 \label{eq:radialmoment}
\end{equation}
\end{lemma}

\begin{proof}
The variable $\norm{h}^2$ has the $\chi_d^2$ distribution, and
\[
 \E\norm{h}^{2n}=2^n\frac{\Gamma(n+d/2)}{\Gamma(d/2)}
 =\prod_{j=0}^{n-1}(d+2j)\le(d+2n)^n.
\]
\end{proof}

\end{document}